\documentclass[a4paper]{llncs} 
\usepackage[T1]{fontenc}
\usepackage[utf8x]{inputenc}
\usepackage[english]{babel}
\usepackage{ucs}
\usepackage{amsmath}
\usepackage{amsfonts}
\usepackage{amssymb}
\usepackage{amscd}
\usepackage{graphicx}
\usepackage[algo2e, linesnumbered, boxruled]{algorithm2e}
\usepackage{setspace}
\usepackage{color}
\usepackage{hyperref}
\usepackage{url}

\newcommand{\R}{\ensuremath{\mathbb{R}}}
\newcommand{\IR}{\mathbb{IR}}

\newcommand{\inter}[1]{\left[#1\right]}
\newcommand{\lb}[1]{\underline{#1}}
\newcommand{\ub}[1]{\overline{#1}}
\newcommand{\wid}[1]{\text{wid}(#1)}
\newcommand{\midpoint}[1]{\text{mid}(#1)}

\newcommand{\fdef}{\ensuremath{f:\mathcal{D} \subseteq\R^m\rightarrow\R^n}}

\DeclareMathOperator{\range}{range}
\DeclareMathOperator{\diag}{Diag}
\DeclareMathOperator{\offdiag}{OffDiag}
\DeclareMathOperator{\interior}{int}

\newcommand{\va}{{\bf a}}
\newcommand{\vb}{{\bf b}}

\newcommand{\vf}{{\bf f}}

\newcommand{\vu}{{\bf u}}
\newcommand{\vv}{{\bf v}}
\newcommand{\vx}{{\bf x}}
\newcommand{\vy}{{\bf y}}
\newcommand{\vt}{{\bf t}}
\newcommand{\vz}{{\bf z}}
\newcommand{\vJ}{{\bf J}}
\newcommand{\vA}{{\bf A}}
\newcommand{\vB}{{\bf B}}

\newcommand{\fr}{{f_{1:r}}}

\newcommand{\eg}{\textit{e.g.}}
\newcommand{\num}[2]{
	\ensuremath{
		{#1:#2}
	}
}

\author{
	Olivier \textsc{Mullier} \and
	Éric \textsc{Goubault} \and
	Michel \textsc{Kieffer} \and
	Sylvie \textsc{Putot}
}

\title{General Inner Approximation of Vector-valued Functions}
\author{Olivier Mullier, Éric Goubault, Michel Kieffer and Sylvie Putot}

\institute{
  CEA Saclay Nano-INNOV,
  CEA LIST,
  Laboratory for the Modelling and Analysis of Interacting Systems,
  Point Courrier 174,
  91191 Gif sur Yvette CEDEX,
  \{Eric.Goubault,Sylvie.Putot\}@cea.fr
}
\begin{document}
\maketitle
\begin{abstract}
	This paper addresses the problem of evaluating a subset of the range of a vector-valued function.
	It is based on a work by Goldsztejn and Jaulin which provides methods based on interval analysis to address this problem when the dimension of the domain and co-domain of the function are equal.
	This paper extends this result to vector-valued functions with domain and co-domain of different dimensions.
	This extension requires the knowledge of the rank of the Jacobian function on the whole domain.
	This leads to the sub-problem of extracting an interval sub-matrix of maximum rank from a given interval matrix.
	Three different techniques leading to approximate solutions of this extraction are proposed and compared.
\end{abstract}		
	
\section{Introduction}
	Computing the values a function can take over some domain is generally of great interest in the analysis of numerical programs as in abstract interpretation~\cite{Cousot1977}, in robust control of dynamic systems~\cite{jaulin1996guaranteed}, or in global optimization~\cite{neumaier2004complete}.
	Computing the image of a domain by a function (also called direct image or \emph{range}) exactly is intractable in general.
	The classical solution is then to compute an outer approximation of this range, which can unfortunately be very pessimistic.
	This outer approximation may introduce many values which do not belong to the range.
	Providing, in addition, an inner approximation can be helpful to state the quality of the outer approximation.
	
	For scalar-valued functions, an inner approximation can be evaluated using modal intervals~\cite{HerreroSVJ05} (using Kaucher arithmetic~\cite{kaucher1980interval})  or twin arithmetic~\cite{nesterov1997}.
	When $f$ maps $\R^n$ to $\R^n$, modal intervals~\cite{goldsztejn2005} can also be used in the linear case.
	For the non-linear case, set inversion~\cite{jaulin2001applied} can be used when $f$ is \emph{globally} invertible (when an inverse function $f^{-1}$ can be produced). 
		
	In the more general case of $f$ being \emph{locally} invertible, the method described by Goldsztejn and Jaulin in \cite{Goldsztejn2010} can be applied.
	This technique requires however the inverse of the Jacobian of $f$.
	Thus it can only be applied for functions from $\R^n$ to $\R^n$ of constant rank $n$.
	
	This paper proposes a generalization of the method in~\cite{Goldsztejn2010} to deal with functions $f$ from $\mathcal{D} \subseteq \R^m$ to $\R^n$, with $m \neq n$, with rank $r$.
	It describes a method to compute an inner approximation for at most $r$ components of $f$.
	As in \cite{Goldsztejn2010}, the evaluation of the Jacobian of the function on a given subset of its domain is needed.
	There, the identification of the components that can be used to compute an inner approximation has to be done by extracting the sub-matrix of full rank in its Jacobian.
	Checking regularity of interval matrices is a NP-hard problem~\cite{poljak1993checking}, so is the problem of extracting an interval sub-matrix of full rank.
	To our knowledge, no necessary and sufficient condition for checking regularity can be used to address this problem (a list of necessary and sufficient conditions for an interval matrix to be regular can be found in~\cite{rohn1989systems}).
	
	This paper is organized as follows: Section~\ref{sec:gold} recalls the main result of \cite{Goldsztejn2010} on the computation of an inner approximation of the range of vector-valued functions with domain and co-domain of the same dimension. 
	Section~\ref{sec:extension} describes the extension of this result to functions with domain and co-domain of different dimensions.
	Section~\ref{sec:algo} addresses algorithms for computing and inner approximation and describes how sub-matrices of full rank can be extracted from a given interval matrix using different techniques.
	The computation of an inner approximation of the range of functions is illustrated on examples in Section~\ref{sec:app}.

\subsection*{\textbf{Notations}}
	$\vx = \inter{\lb{x},\ub{x}} \triangleq \{x \in \R : \lb{x} \leqslant x \leqslant \ub{x}\}$ is an \emph{interval} where $\lb{x}$ and $\ub{x}$ are respectively its lower and its upper bound.	
	$\IR \triangleq \{\inter{\lb{x},\ub{x}} :  \lb{x}, \ub{x} \in \R,  \lb{x} \leqslant \ub{x}\}$ represents the set of intervals.
	A box is the Cartesian product of $n$ intervals in $\IR^n$.
	For an interval $\vx = \inter{\lb{x}, \ub{x}}$, the \emph{width} is $\wid{\vx} \triangleq \ub{x} - \lb{x}$, the \emph{midpoint} is $\midpoint{\vx} \triangleq \frac{1}{2}(\ub{x} + \lb{x})$, the \emph{interior} is $\interior(\vx) \triangleq \{x \in \R | \lb{x} < x < \ub{x} \}$, and the \emph{boundary} is denoted by $\partial \vx$.
	The \emph{magnitude} is denoted $|\vx| \triangleq \max\{|\lb{x}|, |\ub{x}|\}$ and the \emph{mignitude} is $\langle \vx \rangle  \triangleq \min\{|\lb{x}|, |\ub{x}|\}$ if $0 \notin \vx$ and $\langle \vx \rangle = 0$ otherwise.
	The width of an interval vector $\vx \in \IR^n$ is $\max_{1 \leqslant i \leqslant n}(\wid{\vx_i})$.
	
	The core of interval analysis is its fundamental theorem (see, \eg, \cite{moore1966} or \cite{neumaier1990}) asserting that an evaluation of an expression using intervals gives an outer approximation of the range of this expression over the considered intervals.
	An interval function is an inclusion function denoted here $\vf$:  $f(\vx) = \{ f(x): x \in \vx \} \subseteq \vf(\vx)$ for $\vx$ included in the domain of $f$. 
	For an interval square matrix $\vA \in \IR^{n \times n}$, $\diag \vA \in \IR^{n \times n}$ is the diagonal interval matrix whose diagonal entries are $(\diag \vA)_{ii} = \vA_{ii}$, $1 \leqslant i \leqslant n$, and 0 elsewhere.
	$\offdiag \vA \in \IR^{n \times n}$ is the interval matrix with null diagonal and with off-diagonal entries such that $(\offdiag \vA)_{ij} = \vA_{ij}$.	
	For a vector-valued function $\fdef$ and $x \in \mathcal{D}$, $f_{\num{i}{j}}(x)\triangleq (f_i(x), f_{i+1}(x), \dots, f_j(x))^T$ for $i \leqslant j$.
	For the Jacobian $J^f$ of $f$ and $x \in \mathcal{D}$, 
	\begin{equation*}
	J^{f_\num{i}{j}, x_\num{k}{\ell}}(x) \triangleq
	\left(
		\begin{array}{cccc}
			\frac{\partial f_i}{\partial x_k}(x)	&	\frac{\partial f_i}{\partial x_{k+1}}(x)	&	\dots & \frac{\partial f_i}{\partial x_\ell}(x)\\
			\frac{\partial f_{i+1}}{\partial x_k}(x)	&	\frac{\partial f_{i+1}}{\partial x_{k+1}}(x)	&	\dots & \frac{\partial f_{i+1}}{\partial x_\ell}(x)\\			
			\vdots											&	\vdots													&	\vdots\\
			\frac{\partial f_j}{\partial x_k}(x)	&	\frac{\partial f_j}{\partial x_{k+1}}(x)	&	\dots & \frac{\partial f_j}{\partial x_\ell}(x)			
		\end{array}
	\right)
	\end{equation*}	
	is the restriction of the Jacobian of $f$ for $j-i$ components of $f$ and $\ell-k$ components of $x$.
	$I_k \in \R^{k\times k}$ is the identity matrix of dimension $k$.
	The null matrix with $k$ rows and $\ell$ columns is denoted $0_{k\times \ell}$ and the null vector of $k$ entries is denoted $0_k  \triangleq(0, \dots, 0)^T$.

\section{Inner approximation for functions with domain and co-domain of the same dimension}
\label{sec:gold}
	This section recalls the main result of \cite{Goldsztejn2010} to evaluate an inner approximation of the range of a function with domain and co-domain of the same dimension.
\begin{corollary}
	\label{cor:gold}
	Let $\vx \in \IR^n$ and $f : \vx \rightarrow \R^n$ be a continuous function continuously differentiable in $\interior(\vx)$.
	Consider $\vy \in \IR^n$ and $\tilde{x} \in \vx$ such that $f(\tilde{x}) \in \vy$.
	Consider also an interval matrix $\vJ \in \IR^{n\times n}$ such that  $f'(x) \in \vJ$ for all $x \in \vx$.
	Assume that $0 \notin \vJ_{ii}$ for all $i \in \inter{1,\dots,n}$. 
	Let
	\begin{equation}
		H(\vJ, \tilde{x}, \vx, \vy) = \tilde{x} + (\diag^{-1} \vJ)\Big( \vy - f(\tilde{x}) - (\offdiag \vJ)(\vx - \tilde{x})  \Big)
		\label{eq:hf}
	\end{equation}
	If $H(\vJ, \tilde{x}, \vx, \vy) \subseteq \interior(\vx)$ then $\vy \subseteq \range(f, \vx)$.		
\end{corollary}
	This corollary provides an efficient test for a box $\vy$ to be a subset of the range of a vector-valued function.
	It can be used to compute an inner approximation of functions $f$ from $\R^n$ to $\R^n$, see Section~\ref{sec:algo}.
	The restriction on $f$ having same dimension of domain and co-domain comes from the matrix inversion of $\diag \vJ$ in \eqref{eq:hf}.
\begin{figure}
	\centering 
	\def\svgwidth{\columnwidth} 
	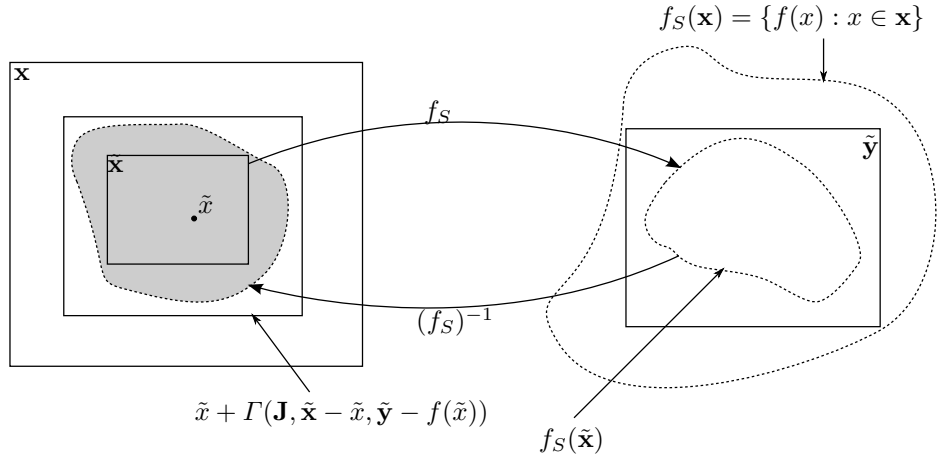
	\caption{Sets and functions involved in Corollary~\ref{cor:gold} for inner approximation.}
	\label{fig:exp}
\end{figure} 
	Figure~\ref{fig:exp} illustrates the computation in Corollary~\ref{cor:gold}.
	The left part of Figure~\ref{fig:exp} represents the domain $\vx$ and the right part the co-domain of $f$.
	The set-valued map $f_S$ is defined from $\mathcal{P}(\R^n)$ (power set of $\R^n$) to $\mathcal{P}(\R^n)$ and returns the set $\{ f(x) : x \in \mathcal{D} \}$ for a given set $\mathcal{D}$, see~\cite{aubin2008set}.
	From a given box $\tilde{\vx} \subset \vx$, one wants to know if the box $\tilde{\vy}$ computed by an inclusion function of $f$ over $\tilde{\vx}$ belongs to the range of $f$ or equivalently if $\tilde{\vy} = \vf(\tilde{\vx})$ is a subset of $f_S(\vx)$.
	If $\tilde{\vx}$ is too large compared to $\vx$, one might have $\tilde{\vy} = \vf(\tilde{\vx}) \nsubseteq f_S(\vx)$.
	To prove that $\tilde{\vy} \subset f_S(\vx)$, it is sufficient to prove that $f^{-1}_S(\tilde{\vy}) \subset {\vx}$.
	The function $H(\vJ, \tilde{x}, \vx, \vy)$ in \eqref{eq:hf} can be seen as an inclusion function for $f^{-1}_S\circ f_S(\vx)$.	

\section{Extension for functions with domain and co-domain of different dimensions}
\label{sec:extension}
	Corollary~\ref{cor:gold} only applies for functions having the same dimension for domain and co-domain.
	It also needs that the determinant of the Jacobian is different from 0.
	Consider now the case of a function $f$ with domain and co-domain of different dimensions.
	In what follows, assume that $\fdef$ is a $C^1$ function of rank greater than or equal to $r$ on $\mathcal{D}$.
	It is assumed that there exist $r$ components $(x_{i_1},\dots,x_{i_r})$ of $x = (x_1, \dots, x_m) \in \mathcal{D} \subseteq\R^m$ and $r$ components $(f_{j_1}(x), \dots, f_{j_r}(x))$ such that 
	\begin{equation}
		\label{eq:detRank}
		\forall x\in\mathcal{D} \subseteq\R^m,\
		\det
		\left(
			\frac{\partial f_{j_k}}{\partial x_{i_\ell}}(x)
		\right)_{1\leqslant k, \ell \leqslant r}
		\neq 0	
	\end{equation}
	Hereafter, without loss of generality, $f$ is considered after the permutation of the $r$ coordinates $(x_{i_1},\dots,x_{i_r})$ and the $r$ coordinates $(f_{j_1}(x), \dots, f_{j_r}(x))$ (this permutation is discussed later in Section~\ref{sub:perm}).
	It means that the Jacobian of $f$ has an $r \times r$ sub-matrix on the upper left such that 
	\begin{equation}
		\forall x\in\mathcal{D} \subseteq\R^m,\ \det
		\left(
			\frac{\partial f_j}{\partial x_{k}}(x)
		\right)_{1\leqslant j,k \leqslant r}
		\neq 0	
		\label{eq:inv}
	\end{equation}	
	
	Theorem~3.1 in \cite{Goldsztejn2010} provides sufficient conditions for a box $\vy$ to be included in the range of a function.
	Theorem~\ref{theo:th} bellow generalizes this characterization by providing sufficient conditions for a box $\vy_1 \in \IR^r$ to be inside the projection on the first $r$ components of the image of $\fdef$ when $f$ verifies \eqref{eq:inv}.
	\begin{theorem}
\label{theo:th}
	Let $f: \mathcal{D} \subseteq \R^m \rightarrow \R^n$ be a $C^1$ function that verifies \eqref{eq:inv}, $\vu \subset \mathcal{D}$ a box in $\IR^m$ and $\vy_1 \in \IR^r$.
	Assume that the two following conditions are satisfied
	\begin{enumerate}\renewcommand{\labelenumi}{(\roman{enumi})}
		\item $\vy_1 \cap \fr(\partial \vu) = \emptyset$;\label{cond:1}
		\item $\fr(\tilde{u}) \in \vy_1$ for some $\tilde{u} \in \vu$,\label{cond:2}
	\end{enumerate}
	then $\vy_1 \subseteq \fr(\vu)$.
\end{theorem}
Before starting the proof the next result is needed.
\begin{lemma}
\label{lem:1}
	Let $f: \mathcal{D} \rightarrow \R^n$ be a $C^1$ function satisfying \eqref{eq:inv} and let $E$ be a compact such that $E \subset \mathcal{D}$.
	Then one has $\partial( f_\num{1}{r}(E) ) \subseteq f_\num{1}{r}(\partial E)$.
\end{lemma}
\begin{proof}
	Consider any $y_1 \in \partial( f_\num{1}{r}(E) )$.
	As $f$ is continuous, $\fr$ is continuous as well.
	Then the image of $E$, compact, by $\fr$ is also compact.
	In particular it is closed so $\partial(\fr(E) )$ is included in $\fr(E)$.
	So there exists $x \in E$ such that $y_1 = f_\num{1}{r}(x)$.
	Now suppose that $x \in \interior E$.
	We now prove that this leads to a contradiction.
	As $x \in \interior E$, there exists $U$ open of $E$ with $x \in U$.
	Because of \eqref{eq:inv}, $\fr$ is a submersion and as submersions are open maps (see \cite{tu2010introduction}), $V = f_{1:r}(U)$ is open in $\R^r$.
	We have $y_1 \in V$ then $y_1 \in \interior f_\num{1}{r}(E)$ which contradicts $y_1 \in \partial( f_\num{1}{r}(E) )$. 
	As a conclusion we have $x \in \partial E$ and eventually, $\partial( f_\num{1}{r}(E) ) \subseteq f_\num{1}{r}(\partial E)$.
\end{proof}	
\begin{proof}[Theorem~\ref{theo:th}]
	$\vu$ is a compact of $\mathcal{D}$ and $f$ is a $C^1$ function that verifies \eqref{eq:inv} so we have, from Lemma~\ref{lem:1}, $\partial( \fr(\vu) ) \subseteq \fr(\partial \vu)$.
	So $\vy_1 \cap \partial(\fr(\vu)) \subseteq \vy_1 \cap \fr(\partial \vu) =\emptyset$
therefore 
\begin{equation}
		\label{eq:empty}
		\vy_1 \cap \partial( f_\num{1}{r}(\vu) ) = \emptyset.
	\end{equation}
	The set $\fr(\vu)$ is compact because $\vu$ is compact and $\fr$ is continuous.
	Let $\tilde{u}$  and $y_1 = \fr(\tilde{u}) \in \vy_1$ be given in ($ii$).
	As the intersection of $\vy_1$ and $\partial \fr(\vu)$ is empty by \eqref{eq:empty}, $y_1 \in \interior \fr(\vu)$.
	Consider any $z \in \vy_1$ and suppose that $z \notin \fr(\vu)$.
	Since $\vy_1$ is path connected, there exists a path included in $\vy_1$ between $y_1$ and $z$ such that, by Lemma~A.1. in \cite{Goldsztejn2010}, this path intersects $\partial \fr(\vu)$ which is not possible from \eqref{eq:empty}.
	Therefore $z \in \fr(\vu)$ which concludes the proof.
\end{proof}

	Theorem~\ref{theo:th} is a generalization of Theorem~3.1 in \cite{Goldsztejn2010} for a function satisfying \eqref{eq:inv}.
	In Theorem~3.1 in \cite{Goldsztejn2010}, the set $\Sigma = \{ x \in \interior{\vx}\ |\ \det f'(x) = 0 \}$ can be extended for $\fr$ by $\Sigma_2 = \{ x \in \interior{\vx}\ |\ \text{rank}(\fr(x)) < r\}$.
	Due to \eqref{eq:inv}, one has $\Sigma_2 = \emptyset$.
	In what follows, Corollary~\ref{cor:rang} of Theorem~\ref{theo:th}, which extends the inclusion test of Corollary~\ref{cor:gold}, is introduced.
	\begin{corollary}
	\label{cor:rang}
	Let $f:\mathcal{D} \subseteq\R^m\rightarrow\R^n$ be a $C^1$ function that satisfies \eqref{eq:inv} and $\vu = (\vu_1, \vu_2)\in \IR^r\times\IR^{m-r}$. 
	Consider $\vy_1 \in \IR^r$, $\tilde{u} = (\tilde{u}_1, \tilde{u}_2) \in (\vu_1, \vu_2)$ such that $\fr(\tilde{u}) \in \vy_1$ and $\vJ^\fr = \left( \vJ^{\fr, u_1}\ \vJ^{\fr, u_2} \right)\in \IR^{r\times m}$ an interval matrix containing $J^\fr$ the Jacobian of $\fr$ on $(\vu_1, \vu_2)$ such that $0 \notin \left(\vJ^{\fr, u_1}\right)_{ii}$ for $1 \leqslant i \leqslant r$. 
	Let
	\begin{eqnarray*}
		&&H_\fr(\vJ^\fr, \tilde{u}, \vu, \vy_1) = \tilde{u}_1 + (\diag^{-1} \vJ^{\fr, u_1}) \times\\
		 &&\Big(\vy_1 - \fr(\tilde{u}) - (\offdiag \vJ^{\fr, u_1})(\vu_1 - \tilde{u}_1) - \vJ^{\fr, u_2}(\vu_2 - \tilde{u}_2)\Big).
	\end{eqnarray*}
	If
	\begin{equation}
		H_\fr(\vJ^\fr, \tilde{u}, \vu, \vy_1) \subseteq \interior(\vu_1),
		\label{eq:rang}
	\end{equation}
	then
	\begin{equation*}
			\vy_1\subseteq \fr(\vu).
	\end{equation*}
\end{corollary}	
\begin{proof}
	It is sufficient to prove that if \eqref{eq:rang} is satisfied, the conditions of Theorem~\ref{theo:th} are satisfied too.
	
	($i$) Let $u = (u_1, u_2) \in \partial \vu$.
	Since $u \in \vu$, the mean value theorem applied to $f_{1:r}$ (see~\cite{neumaier1990}) shows that
	\begin{equation}
		\fr(u) \in \fr(\tilde{u}) + \vJ^{\fr} (\vu - \tilde{u})
	\end{equation}
	Let us show that $\fr(\tilde{u}) + \vJ^{\fr} (\vu - \tilde{u}) \cap \vy_1 \neq \emptyset$ which implies $(i)$ false contradicts \eqref{eq:rang}.
	Assume that  there exists $J \in \vJ^{\fr}$, $J = (J_1 J_2)$ with $J_1 \in \R^{r \times r}$ and $J_2 \in \R^{r \times m - r}$; $u = (u_1, u_2)^T$, $\tilde{u} = (\tilde{u}_1, \tilde{u}_2)^T$, and $y_1 \in \vy_1$ such that
	\begin{eqnarray}
		\label{eq:mv}
		y_1	&=&	\fr(\tilde{u}) + J(u - \tilde{u})\nonumber\\
				&=&	\fr(\tilde{u}) + J_1(u_1 - \tilde{u}_1) + J_2(u_2 - \tilde{u}_2)
				\label{eq:temp}
	\end{eqnarray}
	By splitting $J_1$ in $\diag J_1 + \offdiag J_1$ in~\eqref{eq:temp}, we obtain:
	\begin{equation*}
		y_1 - \fr(\tilde{u}) -  J_2(u_2 - \tilde{u}_2) = (\diag J_1)(u_1 - \tilde{u}_1) + (\offdiag J_1)(u_1 - \tilde{u}_1)
		\end{equation*}
		\begin{equation*}
		\tilde{u}_1 + (\diag^{-1} J_1)\left(y_1 - \fr(\tilde{u}) - (\offdiag J_1)(u_1 - \tilde{u}_1) - J_2(u_2 - \tilde{u}_2)\right) = u_1
		\end{equation*}
	As $u \in \vu$, $y_1 \in \vy_1$, $\diag^{-1} J_1 \in \diag^{-1} \vJ_1$, $\offdiag J_1 \in \offdiag \vJ_1$, and $J_2 \in \vJ_2$, one gets
	\begin{equation}
		u_1 \in  H_\fr((\vJ_1\ \vJ_2), \tilde{u}, \vu, \vy_1) 
	\end{equation}
	and $u_1 \in \partial \vu_1$ which contradicts \eqref{eq:rang}.
	Then \eqref{eq:rang} implies $(i)$.

	$(ii)$ By hypothesis, $\fr(\tilde{u}) \in \vy_1$.
\end{proof}
	
	For a function $\fdef$, Corollary~\ref{cor:rang} gives a test for a box to belong to the image of $r$ components of $f$. 	
	It can only be performed if the Jacobian for $r$ components of the function evaluated over the considered box is of full rank $r$.		
	When the rank of $f$ equals the dimension of the co-domain, $f$ is a \emph{submersion}~\cite{arnold1985}, Corollary~\ref{cor:rang} can be used to compute an inner approximation of the entire range of $f$.

\begin{example}
	Let $f : \mathcal{D} \subset \R \rightarrow \R^3$ be a function that satisfies \eqref{eq:inv} for $r = 1$.
	\begin{eqnarray}
		f :	&&	\vx \subset \R \rightarrow \R^3\nonumber\\
			&&	x \mapsto
			\left(
				\begin{array}{c}
					\sin 2x\\
					\sin x\\
					\frac{x}{2}
				\end{array}
			\right)
			\label{eq:ex1}
	\end{eqnarray}	
	The box $\vx = \inter{0, \pi} \subset \mathcal{D}$ is considered as the domain on which $f$ is studied.
	The function is of constant rank 1 then using Corollary~\ref{cor:rang}, one is able to compute an inner approximation of the range of a single component of $f$ e.g. $f_1(\vx)$, $f_2(\vx)$ or $f_3(\vx)$.
	\begin{figure} 
		\centering 
		\def\svgwidth{.75\columnwidth} 
		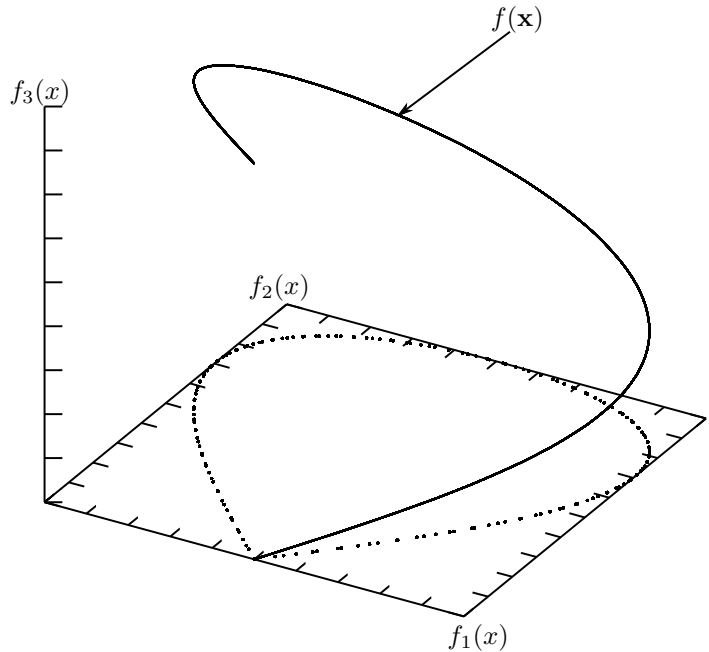
		\caption{Example for $f: (\vx) \subseteq \R \mapsto \R^3$ of constant rank 1.}
		\label{fig:gen}
	\end{figure}	
	There is of course no proper box of dimension 2 or 3 included in the range of $f$.
	Figure~\ref{fig:gen} represents the range of the function defined in \eqref{eq:ex1}.
\end{example}	
\section{Algorithms}	
\label{sec:algo}
\subsection{When domain and co-domain have the same dimension}
	Algorithm~1 in \cite{Goldsztejn2010} computes an inner approximation for functions with domain and co-domain of the same dimension using Corollary~\ref{cor:gold} and a bisection algorithm.
	The method is as follows.
	For a given box $\tilde{\vx}$ included in the initial domain $\vx$, a box $\tilde{\vy}$ such that $f(\vx) = \{ f(x) : x \in \tilde{\vx} \} \subseteq \tilde{\vy}$ is computed using the interval extension $\vf$ of $f$.

	If the hypotheses of Corollary~\ref{cor:gold} are satisfied, $\tilde{\vy}$ is part of an inner approximation of the range of $f$.
	If they are not satisfied, $\tilde{\vx}$ is partitioned into two smaller boxes $\tilde{\vx}'$ and $\tilde{\vx}''$ that are treated like $\tilde{\vx}$ was.
	If the box $\tilde{\vx} = (\tilde{\vx}_1, \dots, \tilde{\vx}_m)$ is deemed too small to be further bisected (i.e. when $\wid{\vx} < \varepsilon$ where $\varepsilon$ is a user-defined parameter), then the iterations stop for this box.
	This is described in Algorithm~\ref{algo:bis}.
	It uses the function Inner described in Algorithm~\ref{algo:inner}, to decide if a box belongs to the range of a function.
	\newcommand{\linside}{\mathcal{L}_{\texttt{Inside}}}
\newcommand{\lboundary}{\mathcal{L}_{\texttt{Boundary}}}
\newcommand{\ldomain}{\mathcal{L}_{\texttt{Domain}}}

\begin{algorithm2e}
\label{algo:bis}
	\caption{Range inner-approximation evaluation algorithm as defined in \cite{Goldsztejn2010}.}
	\KwIn{$f$, $\vx$, $\epsilon$} 
	\KwOut{$\linside$ (list of boxes), $\lboundary$ (list of boxes)}
	$\linside$: empty list of boxes\;
	$\ldomain$: empty list of boxes (sorted by decreasing width)\;	
	$\ldomain \leftarrow \ldomain \cup \{ \vx \}$\; 
	\While{$\ldomain$ not empty}{
		$\tilde{\vx} \leftarrow$ Extract($\ldomain$)\;
		$\tilde{\vy} \leftarrow \vf(\vx) \cap (\vf(\midpoint{\tilde{\vx}}) + \vf'(\tilde{\vx})(\tilde{\vx} - \midpoint{\tilde{\vx}}))$;\tcp{Inclusion function}
		\uIf{\emph{Inner}$(f$, $\vx$, $\tilde{\vx}$, $\tilde{\vy})$}
		{\tcp{Inner is from Algorithm~\ref{algo:inner} or Algorithm~\ref{algo:inner2} according to the dimension of domain and co-domain}
			$\linside \leftarrow \linside \cup \{ \tilde{\vy} \}$\; 
		}
		\uElseIf{\emph{wid}$(\tilde{\vx}) \ge \epsilon$}
		{
			Bisect the box $\tilde{\vx}$ to obtain $\tilde{\vx}'$ and $\tilde{\vx}''$\;
			$\ldomain \leftarrow \ldomain \cup \{ \tilde{\vx}', \tilde{\vx}'' \}$\;
		}
		\Else
		{
			$\lboundary \leftarrow \lboundary \cup \{ \tilde{\vy} \}$\;
		} 				
	}
	\Return{$(\linside, \lboundary)$}\;	
\end{algorithm2e}	
	
	Algorithm~\ref{algo:inner}	decides for a given box $\tilde{\vx} \subset \vx$ whether $\vf(\tilde{\vx})$ belongs to the range of $f$ over $\vx$.
	The parameters $\tau$ and $\mu$ are used for the domain inflation (see Section~5.2 in \cite{Goldsztejn2010}) and $C$ is used to precondition the interval matrix $\vJ^f$ (see Section~4 in \cite{Goldsztejn2010}).
	\begin{algorithm2e}
	\KwIn{$f$, $\vx$, $\tilde{\vx}$, $\vy$} 
	\KwOut{Boolean}
	$\tau \leftarrow 1.01$; $\mu \leftarrow 0.9$\;
	$k \leftarrow 0$\;
	$\tilde{x} \leftarrow \midpoint{\tilde{\vx}}$\;
	$C \leftarrow f'(\tilde{x})^{-1}$\;
	$\vb \leftarrow C\tilde{y} - Cf(\tilde{x})$\;
	$d_k \leftarrow + \infty$; $d_{k-1} \leftarrow + \infty$\;
	\While{$d_k \le \mu d_{k-1} \wedge \tilde{\vx} \subseteq \vx$}
	{
		$\vu \leftarrow \Gamma(C\vf'(\tilde{\vx}),\tilde{\vx}, - \tilde{x}, \vb)$;\tcp{$\Gamma$ is as defined in Corollary~\ref{cor:gold}}
		\lIf{$\tilde{x} + \vu \subseteq \tilde{\vx}$}
		{
			\Return{true}\;
		}
		$d_{k-1} \leftarrow d_k$\;
		$d_k \leftarrow d(\tilde{\vx}, \tilde{x} + \tau \vu)$;\tcp{Domain inflation (Cf. Section~5.2 in \cite{Goldsztejn2010})}
		$\tilde{\vx} \leftarrow \tilde{x} + \tau \vu$\;
	}
	\Return{false}\;
	\caption{Inclusion test as defined in \cite{Goldsztejn2010}.}
	\label{algo:inner}
\end{algorithm2e}

\begin{example}
	Let $f(x) = Ax$ with $A =
	\left(
		\begin{array}{rc}
			1 & 1\\
			-1 & 1
		\end{array}
	\right)\text{,}
	$
	and an initial domain $\vx = \left( \inter{-2, 2}, \inter{-2, 2} \right)$.
	The aim is to compute an inner approximation of the set $\{ f(x) : x \in \vx \}$.
	Of course, in this too simple case, direct methods would be applicable since $A$ is an invertible matrix, but this is intended to exemplify the method.
	\begin{figure}
		\centering
		\includegraphics[scale=0.45]{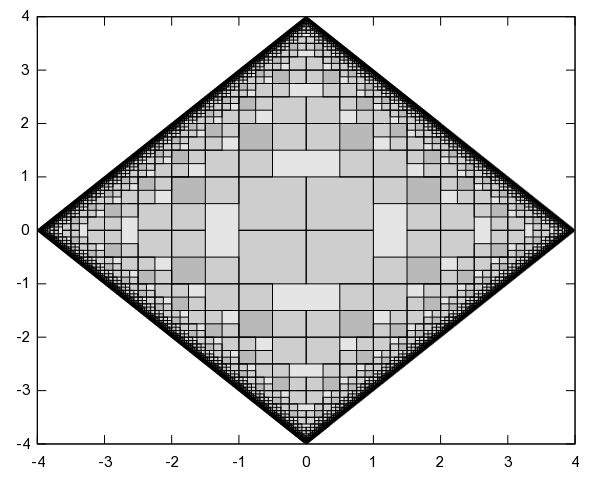}
		\caption{Inner approximation of the range of $f(x) = Ax$ when $x \in \inter{-2, 2}^2$ : $\lboundary$ (in black) and $\linside$ are evaluated using Algorithm~\ref{algo:bis} with $\varepsilon = 10^{-3}$}	
		\label{fig:ax}
	\end{figure}
	Figure~\ref{fig:ax} shows the result obtained using Algorithm~\ref{algo:bis}.
	Since bisections occur in the domain, the result consists of a set of overlapping boxes, obtained by an inclusion function computing outer approximations.
	Dark areas in Figure~\ref{fig:ax} indicate many overlapping boxes.
\end{example}
\subsection{When the domain and co-domain have different dimensions}
	We now extend the method in \cite{Goldsztejn2010} to compute an inner approximation of the projection on $r$ components of $f$.
	Algorithm~\ref{algo:bis} is used unchanged, except for the inner inclusion test Inner in Line~7 which is now implemented by Algorithm~\ref{algo:inner2} instead of Algorithm~\ref{algo:inner}.
	\begin{algorithm2e}
	\KwIn{$\fdef$, $\vx$, $\vx'$, $\vy, r$\tcp{$r$ is the rank of $f$}} 
	\KwOut{Boolean}
	$\tau \leftarrow 1.01$; $\mu \leftarrow 0.9$\;
	$k \leftarrow 0$\;
	$\tilde{x} \leftarrow \midpoint{\tilde{\vx}}$\;
	$\vJ \leftarrow f'(\vx)$\;
	$J^{\fr} \leftarrow$ transformation$(f'(\tilde{x}), r)$; \tcp{recovery of the $r$ first lines of the Jacobian of $f$}
	$C \leftarrow \text{ precondition}(J^\fr)$;\tcp{Computation of the precondition matrix}
	$(\tilde{u_1}, \tilde{u_2}) \leftarrow$ Extract$(\tilde{x}, r, m)$\;
	$(\tilde{\vu_1}, \tilde{\vu_2}) \leftarrow$ Extract$(\tilde{\vx}, r, m)$\;
	$(\vu_1, \vu_2) \leftarrow$ Extract$(\vx, r, m)$\;
	$(\vJ^{\fr,u_1}, \vJ^{\fr,u_2}) \leftarrow$ Extract$(C\vJ, r, m)$\;
	$\vb \leftarrow \left(C\tilde{y} - Cf(\tilde{x})\right)_{1:r}$;\tcp{We need only the $r$ first components}
	$d_k \leftarrow + \infty$; $d_{k-1} \leftarrow + \infty$\;
	\While{$d_k \le \mu d_{k-1} \wedge \tilde{\vu} \subseteq \vu$}
	{
		$\vt \leftarrow (\text{Diag }^{-1}\vJ^{\fr,u_1})(\vy_1 - \fr(\tilde{u}_1, \tilde{u}_2)-(\text{OffDiag }\vJ_{\fr,u_1})(\tilde{\vu}_1-\tilde{u}_1)-\vJ^{\fr,u_2}(\tilde{\vu}_2-\tilde{u}_2))$\;
		\lIf{$\tilde{u} + \vt \subseteq \tilde{\vu}$}
		{
			\Return{true}\;
		}
		$d_{k-1} \leftarrow d_k$\;
		$d_k \leftarrow d(\tilde{\vu}, \tilde{u} + \tau \vt)$;\tcp{Domain inflation (Cf. Section~5.2 in \cite{Goldsztejn2010})}
		$\tilde{\vu} \leftarrow \tilde{u} + \tau \vt$\;
	}
	\Return{false}\;
	\caption{Inclusion test based on Corollary~\ref{cor:rang}}
	\label{algo:inner2}
\end{algorithm2e}
	When using Algorithm~\ref{algo:inner2}, the vector-valued function $f$ is assumed to satisfy \eqref{eq:inv}.
	The main difference with the method in \cite{Goldsztejn2010} is in the construction of the variables needed in the application of Corollary~\ref{cor:rang}: in Algorithm~\ref{algo:inner2}, Lines 7--10 are dedicated to the definition of the vectors $(u_1, u_2)$, the interval vectors $(\vu_1, \vu_2)$ and the interval matrices $(\vJ^{\fr, u_1}$, $\vJ^{\fr, u_2})$ from Corollary~\ref{cor:rang}.
	First, the $r$ components must be separated from the others to obtain $(u_1, u_2)$.
	In Line~7, we construct from a vector in $\mathcal{D} \subseteq\R^m$, the initial domain, a vector in $\R^r\times \R^{m - r}$. 
	\begin{equation*}
		(x_1, \dots, x_m)\mapsto ( (x_1, \dots, x_r), (x_{r+1}, \dots, x_{m}) )
	\end{equation*}
	Lines~8 and 9 construct the same information as at Line~7 but for $\vx \in \IR^m$, an interval vector instead of a vector in $\R^m$, to get $(\vu_1, \vu_2) \in \IR^r \times \IR^{m - r}$.
	In Line~10, the pair of interval matrices $(\vJ^{\fr, u_1}$, $\vJ^{\fr, u_2})) \in \R^{r \times r} \times \R^{r\times m - r}$ are obtained from an interval matrix $\vJ^\fr \in \R^{r \times n}$.

	\paragraph{Preconditioning}
	In~\cite{Goldsztejn2010}, the function has the same dimension for domain and co-domain and the Jacobian is then a square matrix.
	This interval square matrix which is an outer approximation of the Jacobian has to be preconditioned in order to apply the test in Corollary~\ref{cor:gold} with an H-matrix (see Definition~\ref{def:hmat} and Section~4 in \cite{Goldsztejn2010}).
	Here we need also to extend this preconditioning operation.
	In practice, the preconditioning matrix is computed as follows: 
	For a given box $\vu \in \IR^m$, $J^\fr(\tilde{u})$, the Jacobian of the $r$ first components of $f$, is computed for $\tilde{u} = \midpoint{\vu}$ and supplemented with the $(m-r)$ last lines of the identity matrix $I_{m}$ to obtain an $m \times m $ matrix 
		\begin{equation*}
			D_{\tilde{u}} = 
			\left(
				\begin{array}{c|c}
					\multicolumn{2}{c}{J^\fr(\tilde{u})}\\
					0_{m-r \times r}	&	I_{m-r}
				\end{array}
			\right).
		\end{equation*}
	The inverse of $D_{\tilde{u}}$ is computed and its $r$ first columns are extracted to be the preconditioning matrix $C$.
	Decomposing $C$ into $(C_1, C_2)^T$ with $C_1 \in \R^{r \times r}$ and $C_2 \in \R^{m-r \times r}$, the test in Corollary~\ref{cor:rang} becomes
	\begin{equation}
		H_\fr(C\vJ^\fr, \tilde{u}, \vu, C\vy_1) \subseteq \interior(\vu_1).
		\label{eq:rang2}
	\end{equation}
\subsection{Extracting the sub-matrix of maximum rank from an interval matrix}
	\label{sub:perm}
	The use of Algorithm~\ref{algo:inner2} requires that the rank $r$ of the Jacobian of $f$ is known and that $f$ satisfies \eqref{eq:inv}.
	The Jacobian matrix is an interval matrix containing the Jacobian of $f$ over some box.
	In the general case, we thus need to extract an interval sub-matrix of constant rank from the Jacobian of $f$.
	
	In this section, we first define the rank of an interval matrix.
	Then, we propose different methods to extract sub-matrices of full rank from a given interval matrix.
	Some results on the evaluation of the eigenvalues of an interval matrix are well documented (see, e.g., \cite{rohn1993interval}) but are not tractable for our problem.
	The extraction of an $r\times r$ sub-matrix of full rank is also not tractable.
	Thus, we chose to rely on three more tractable - though more approximate - methods aiming at extracting a sub-matrix of high rank from a given interval matrix.
	\begin{definition}[Regular interval square matrix~\cite{rohn1989systems}]\ \\
		Let $\vA \in \IR^{n \times n}$ be an interval matrix.
		$\vA$ is regular if and only if for all matrix $A \in \vA$, $A$ is not singular.
	\end{definition}
	\begin{definition}[Rank of an interval matrix~\cite{ishida1996reasoning}]\ \\
		Let $\vA \in \IR^{n \times m}$ be an interval matrix.
		$\vA$ is of constant rank $r$ if and only if the largest regular interval square sub-matrix $\vA_0$ of $\vA$, is of dimension $r$.
		\label{def:rankinter}
	\end{definition}
	
	Definition~\ref{def:rankinter} means that for all $A \in \vA \in \IR^{n\times m}$, the rank of $A$ is larger than or equal to $r$.
	To extract a regular interval square matrix of dimension equal to the rank of $\vA$, three techniques are proposed in what follows.
	
\subsubsection{Building strictly dominant interval sub-matrices}
	\label{sec:sdd}

	This first method relies on the Levy-Desplanques theorem on strictly dominant matrices as a simple test for non-singularity.
	We uses this test to formulate the extraction of sub-matrices of full rank as a linear programming problem.
	\begin{definition}[Strictly diagonally dominant matrix~\cite{golub1996}]\ \\
		Let $A = \left(a_{ij}\right)_{1\leqslant i,j \leqslant m}\in\R^{m\times m}$ be a square matrix.
		$A$ is a strictly diagonally dominant matrix if and only if
		\begin{equation}
			\forall i\in \{1,\dots,m\}, |a_{ii}|> \sum_{\substack{i=1\\i\neq j}}^m |a_{ij}|.
		\end{equation}
	\end{definition}
	This definition can be extended to interval matrices, using magnitude and mignitude instead of the absolute value:
	\begin{definition}[Strictly diagonally dominant interval matrix~\cite{neumaier1990}]\ \\
		Let $\vA = \left(\va_{ij}\right)_{1\leqslant i,j \leqslant m}\in \IR^{m\times m}$ be a square interval matrix.
		$\vA$ is a strictly diagonally dominant interval matrix if and only if
		\begin{equation}
			\forall i\in \{1,\dots,m\}, \langle\va_{ii}\rangle> \sum_{\substack{i=1\\i\neq j}}^m |\va_{ij}|
			\label{eq:diag}
		\end{equation}
		\label{def:sdd}
	\end{definition}
	\begin{theorem}[Levy-Desplanques theorem~\cite{levy1881,taussky1949recurring}]
		\label{theo:levy}
		A strictly diagonally dominant (interval) matrix is regular.
	\end{theorem}

	Consider $\vA \in \IR^{n\times m}$, an interval matrix.
	We introduce the decision variables $x_{ij}$ with $1 \leqslant i \leqslant n$ and $1 \leqslant j \leqslant m$.
	The boolean $x_{ij}$ equals 1 if the component $\va_{ij}$ of $\vA$ is picked to be an element of the diagonal of the rank $k$ sub-matrix of $\vA$ and 0 otherwise.
	The $x_{ij}$s are obtained as solutions of the following constrained optimization problem
	\begin{equation}
		\begin{aligned}
			\max\	&	f(x) = 	 \sum_{i=1}^n\sum_{j=1}^m x_{ij}	&\\
	  		s.t.		& 
	  		\left\{
	  			\begin{array}{lr}
	  					\displaystyle{\sum_{i=1}^n x_{ij}} \leqslant 1		&	j = 1,\dots,m\\
	     				\displaystyle{\sum_{j=1}^m x_{ij}} \leqslant 1	&	i = 1,\dots,n\\
						\displaystyle{\sum_{k = 1}^n x_{kl} \ \langle \va_{kl} \rangle > \sum_{i = 1}^n \sum_{j = 1}^m x_{ij} (1 - x_{il}) |\va_{il}|}	&	\begin{array}{c}k = 1,\dots,n\\l = 1,\dots,m\end{array}\\								
						x_{ij}\in \{0,1\}	&
     			\end{array}
     		\right.
	 	\end{aligned}
      	\label{eq:lp}
    \end{equation}	
    The objective is to maximize the size of a square regular sub-matrix of $\vA$.	
   	The two first constraints ensure that at most one component on each row and column of the interval matrix $\vA$ is taken (it corresponds to the problem of placing towers in a possibly not square chess board).
   	The last constraint corresponds to Theorem~\ref{theo:levy}.
   	Figure~\ref{fig:solpl} shows an example of solution provided by the constrained optimization problem \eqref{eq:lp}. 
   	\begin{figure}
   	\centering
   		\begin{tabular}{|c|c|c|c|c|}
   			\hline
   				\;\;\; &	$\bigstar$	&	$\blacklozenge$	&	$\blacklozenge$	& \;\;\;	\\
   			\hline	
   				\ &	$\blacklozenge$	&	$\blacklozenge$	&	$\bigstar$	&	\\
   			\hline	
   				\ &	&	&	&	\\
   			\hline	
   				&	$\blacklozenge$	&	$\bigstar$	&	$\blacklozenge$	&	\\
   			\hline
   		\end{tabular}
   		\caption{Example of result provided by the method using strictly dominance (see Section~\ref{sec:sdd}) and the one using H-matrices (see Section~\ref{sec:hmat}) on an interval matrix $\vA \in \IR^{4 \times 5}$: $\bigstar$ represents components of the matrix that have been chosen ($x_{ij} = 1$) for the diagonal and $\blacklozenge$ represents the non-diagonal entries of the sub-matrix. Empty boxes represent components that are not part of the sub-matrix.}
   		\label{fig:solpl}
   	\end{figure} 
   	A component of the interval matrix is picked if and only if it satisfies \eqref{eq:diag} and then leads to a strictly diagonally dominant interval matrix.
   	
   	The last constraints in \eqref{eq:lp} are quadratic and have to be turned into linear constraints for efficiency reasons since linear programming techniques are generally fast.
   	A given $\va_{kl}$ is chosen if the sum of all the other $\va_{il}$ for $i = 1, \dots, m$ for which there exists an $\va_{ij}$ that is part of the diagonal of the extracted sub-matrix is lower.
   	Equivalently, 
	\begin{equation}
		x_{kl} = 1 \Rightarrow \langle \va_{kl}\rangle > \sum_{\substack{i=1\\i\ne k}}^n\sum_{j=1}^m x_{ij}|\va_{il}|.
	\end{equation}	
	Using the so called Big-M relaxation (see, e.g., \cite{hooker2011integrated}), this constraint can be rewritten as follows.
	\begin{equation}
		\sum_{\substack{i=1\\i\ne k}}^n\sum_{j=1}^m x_{ij}|\va_{il}| \leqslant M + (\langle\va_{kl}\rangle - \mu - M)x_{kl}
		\label{eq:constraint}
	\end{equation}
	with $M$ chosen to be larger than $\sum_{\substack{i=1\\i\ne k}}^n \sum_{j=1}^m x_{ij}|\va_{il}|$ in order to deactivate the constraint when $x_{kl} = 0$ and $\mu$ as small as possible to approximate the strict inequality but not too small to avoid introduction of numerical instability.
	Using~\eqref{eq:constraint} in~\eqref{eq:lp}, the constrained optimization problem~\eqref{eq:lp} becomes 
	\begin{equation}
		\begin{aligned}
			\max\	&	f(x) = 	 \sum_{i=1}^n\sum_{j=1}^m x_{ij}	&\\
	  		s.t.		& 
	  		\left\{
	  			\begin{array}{lr}
	  					\displaystyle{\sum_{i=1}^n x_{ij}} \leqslant 1		&	j = 1,\dots,m\\
	     				\displaystyle{\sum_{j=1}^m x_{ij}} \leqslant 1	&	i = 1,\dots,n\\
						\displaystyle{\sum_{\substack{i=1\\i\ne k}}^n\sum_{j=1}^m} x_{ij}|\va_{il}| \leqslant M + (\langle\va_{kl}\rangle+ \mu - M)x_{kl}	&	\begin{array}{c}k = 1,\dots,n\\l = 1,\dots,m\end{array}\\
						x_{ij}\in \{0,1\}	&
     			\end{array}
     		\right.
	 	\end{aligned}
      	\label{eq:lpsdd}
    \end{equation}	
	Using a linear programming solver on~\eqref{eq:lpsdd}, a strictly dominant interval matrix can be extracted from $\vA$.
	The property for an interval matrix $\vA$ to be a strictly dominant interval matrix is on the rows of $\vA$.
	This definition can apply also for $\vA^T$ the transpose of $\vA$, this is why the linear program is solved for both $\vA$ and $\vA^T$ to obtain the best result.

\subsubsection{Building H-sub-matrices}
\label{sec:hmat}
A second method is now investigated.
It uses a generalization of strictly dominant interval matrices, i.e., the notion of H-matrices~\cite{neumaier1990}.
Basic results on H-matrices are first provided before showing the slight changes in the constraint~\eqref{eq:constraint} that have to be done in order to detect H-sub-matrices in an interval matrix.

\begin{definition}[Comparison Matrix~\cite{neumaier1990}]\ \\
	Let $\vA \in \R^{m \times m}$ be a square interval matrix. 
	The comparison matrix $\langle \vA \rangle$ is built as follows
	\begin{equation*}
		\langle \vA \rangle_{ij} =
		\begin{cases}
			\langle \vA_{ij} \rangle \text{ if } i = j\\
			-|\vA_{ij}| \text{ otherwise}
		\end{cases}\text{ with } i, j = 1, \dots, m.
	\end{equation*}
\end{definition}

\begin{definition}[H-matrix~\cite{neumaier1990}]\ \\
	Let $\vA \in \R^{m \times m}$ be a square interval matrix.
	$\vA$ is an H-matrix if and only if there exists $u > 0_{m}$ such that $\langle \vA \rangle u > 0_m$.
	\label{def:hmat}
\end{definition}

\begin{theorem}[\cite{neumaier1990}]
	Every H-matrix is regular.
\end{theorem}

\begin{remark}
	The notion of H-matrices generalizes the one of strictly dominant interval matrices since a strictly diagonally dominant interval matrix is a particular case of an H-matrix by fixing $u = (\underbrace{1, \dots, 1}_{m})^T$ in Definition~\ref{def:hmat}.
	\label{rmk:h}
\end{remark}

From Remark~\ref{rmk:h}, only slight changes have to be done in order to detect an H-matrix instead of a strictly diagonally one. 
The constrained optimization Problem~\eqref{eq:lpsdd} is transformed into
\begin{equation}
		\begin{aligned}
			\max\	&	f(x) = 	 \sum_{i=1}^n\sum_{j=1}^m x_{ij}	&\\
	  		s.t.		& 
	  		\left\{
	  			\begin{array}{lr}
	  					\displaystyle{\sum_{i=1}^n x_{ij}} \leqslant 1		&	j = 1,\dots,m\\
	     				\displaystyle{\sum_{j=1}^m x_{ij}} \leqslant 1	&	i = 1,\dots,n\\
						\displaystyle{\sum_{\substack{i=1\\i\ne k}}^n\sum_{j=1}^m} x_{ij}|\va_{il}|u_{ij} \leqslant M + (\langle\va_{kl}\rangle u_{kl} - \mu - M)x_{kl}	&	\begin{array}{c}k = 1,\dots,n\\l = 1,\dots,m\end{array}\\								
						x_{ij}\in \{0,1\}	&\\
						u_{ij}> 0	&						
     			\end{array}
     		\right.
	 	\end{aligned}
      	\label{eq:lp2bis}
    \end{equation}	
    Figure~\ref{fig:solpl} shows an example of solution provided by the constrained optimization Problem \eqref{eq:lp2bis}.
	In \eqref{eq:lp2bis}, the last constraint requires a matrix of variables $U = (u_{ij})_{\substack{1\leq i \leq n\\1 \leq j \leq m}}$ to be introduced.
	It corresponds to the vector $u$ in Definition~\ref{def:hmat}.
	This means we now have to solve a quadratic problem that could be tackled using SDP solvers.
	In order to solve \eqref{eq:lp2bis} as efficiently as possible, i.e., by using linear programming techniques, we thus chose a particular $u$ before solving~\eqref{eq:lp2bis}.
	All components of $u$ are chosen to be the inverse of the mignitude of the diagonal entries of the considered sub-matrix (as recommended in~\cite{neumaier1990}): $u = (\langle \va_{ij} \rangle^{-1})_{1\leq i \leq n; 1 \leq j \leq m}$.
The linear program is then
\begin{equation}
		\begin{aligned}
			\max\	&	f(x) = 	 \sum_{i=1}^n\sum_{j=1}^m x_{ij}	&\\
	  		s.t.		& 
	  		\left\{
	  			\begin{array}{lr}
	  					\displaystyle{\sum_{i=1}^n x_{ij}} \leqslant 1		&	j = 1,\dots,m\\
	     				\displaystyle{\sum_{j=1}^m x_{ij}} \leqslant 1	&	i = 1,\dots,n\\
						\displaystyle{\sum_{\substack{i=1\\i\ne k}}^n\sum_{j=1}^m} x_{ij}|\va_{il}|\langle \va_{ij} \rangle^{-1} \leqslant M + (\underbrace{\langle\va_{kl}\rangle \langle \va_{kl} \rangle^{-1}}_{=1} - \mu - M)x_{kl}	&	\begin{array}{c}k = 1,\dots,n\\l = 1,\dots,m\end{array}\\								
						x_{ij}\in \{0,1\}	&
     			\end{array}
     		\right.
	 	\end{aligned}
      	\label{eq:lp2}
    \end{equation}	
	\begin{remark}
		Note that $\langle \va_{kl} \rangle$ has to be different from 0 because of the division that occurs in the constraint.
	\end{remark}
	As in the method using strictly diagonally dominant matrices, the linear program~\eqref{eq:lp2} can be solved using a linear programming solver.
\subsubsection{Combinatorial method}
A random search for an interval sub-matrix of maximum rank is performed.
It could rely on the two previous conditions (Definition~\ref{def:sdd} or Definition~\ref{def:hmat}) to determine whether a matrix is of full rank.
However, since no linear programming formulation has to be consider, one may use a more sophisticated test for full rank verification.
We use for that a result provided in~\cite{rohn1989systems}.
\begin{theorem}[Corollary 5.1 in \cite{rohn1989systems}]
	Let $\vA \in \IR^{m \times m}$ be an square interval matrix. 
	Let $\Delta$ be a matrix such that $\vA = \midpoint{\vA} + \inter{-\Delta, \Delta}$.
	 Let $D = |\text{\emph{mid}}(\vA)^{-1}|\Delta$.
	 If the spectral radius $\rho(D) < 1$ then $\vA$ is regular.
	 \label{thm:regular}
\end{theorem}

	We combine this criterion derived from Theorem~\ref{thm:regular} by extracting randomly chosen components of an interval matrix and testing whether the resulting sub-matrix is regular.
	This process is described in Algorithm~\ref{algoCombi}.
	\begin{algorithm2e}
	\label{algoCombi}
	\caption{Extraction of regular interval sub-matrix}
	\KwIn{$\vA \in \R^{n,m}$}
	\KwOut{$\vB \in \IR^{k,k}$ a regular interval matrix}
	\If{$n = m$ and $\vA$ is regular}
	{
		\Return{\vA}
	}
	\For{$k = \min(n, m)$ \emph{\textbf{downto}} 1}
	{
		\For{$i = 1$ \emph{\textbf{to}} MAX\_ITERATION}
		{
			$\vB \leftarrow$ extraction$(\vA, k)$\tcp{Extraction takes randomly $k$ components of $\vA$ to be in the diagonal of $\vB$ and the other components are deducted from this diagonal.}
			\If{$\vB$ is regular}
			{
				\Return{\vB}
			}
		}
	}
\end{algorithm2e}
\subsubsection{Experiments on the sub-matrix extraction}
	In this section, some results on extracting an interval square sub-matrix of maximum rank from a given interval matrix are now described for the three methods that have been previously described.	
	Two types of experiments have been performed depending on how the considered interval matrix has been produced.	
	The linear programs for the first two methods have been solved using the GLPK interface for C++~\cite{makhorin2006glpk}.
	All experiments have been done on a 2.3 Ghz Intel core i5 processor based laptop with 8 GBytes memory.
	In all experiments, the constant for the linear programs~\eqref{eq:lpsdd} and~\eqref{eq:lp2} are $M = \sum_{i=1}^m \sum_{j=1}^n \ub{\va}_{ij}$ and $\mu = 10^{-2}$.
	For the random extraction, Algorithm~\ref{algoCombi} has been used with a fixed MAX\_ITERATION equal to 500.
	Results have been averaged over 200 realizations.
	
	First experiments have been done on an interval matrix generated randomly but containing  a strictly dominant interval sub-matrix with a fixed dimension.
	The matrix is constructed as follows: the size $m$ of an interval square matrix $\vA$ is chosen.
	For each component $\va_{ij}$ of $\vA$, $\va_{ij} = \inter{\lb{\va}_{ij}, \ub{\va}_{ij}}$ with $1 \leqslant i, j \leqslant m$, $\lb{\va}_{ij}$ is a (pseudo) random number in $\inter{0,9}$ and $\ub{\va}_{ij}$ is equal to $\lb{\va}_{ij} + 1$ for all $i, j = 1, \dots, m$.
	An a priori rank $r$ is chosen.
	Then $r$ coordinates $(i, j)$ are randomly picked and for each of these pairs, the associated interval $\va_{ij} = \inter{\lb{\va}_{ij}, \ub{\va}_{ij}}$ is taken as
	\begin{equation*}
			 \lb{\va}_{ij} = 1 + \sum_{k = 1}^m \ub{\va}_{kj}\text{ and }\ub{\va}_{ij} = \lb{\va}_{ij} + 1
	\end{equation*}
	
	Using this construction, there is in the resulting interval matrix $\vA$ an $r \times r$ interval sub-matrix of $\vA$ and $r$ is a lower bound of the actual rank of $\vA$.
	
	\begin{figure}
	\begin{scriptsize} 
		\centering 
		\def\svgwidth{.8\columnwidth} 
		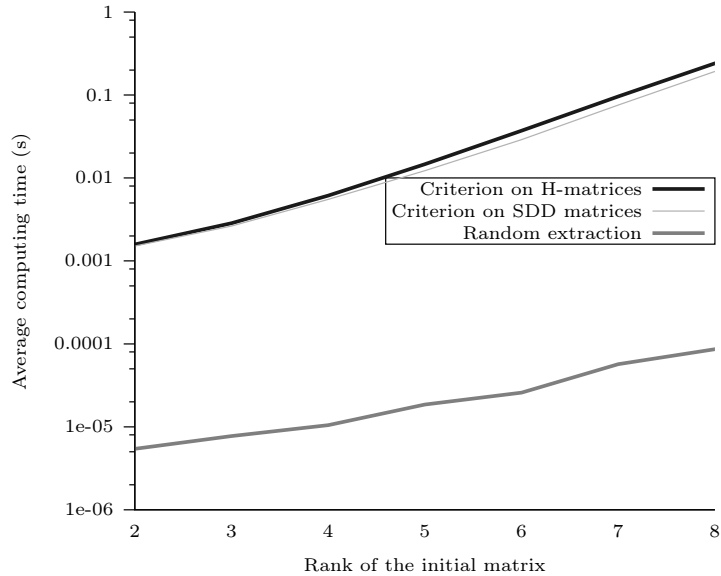
		\caption{Computing time for the LP solver and the random extraction for a full rank interval matrix of size $i\times i$.}
		\label{fig:times1}
	\end{scriptsize}
	\end{figure}	
	Figure~\ref{fig:times1} depicts a first experiment showing the average execution time as a function of the dimension of the considered interval matrix for the three methods.
	The constructed matrices here are square, have a dimension from 2 to 8 and are strictly dominant interval matrices ($r = m$).
	This experiment shows the exponential increase of the computing time needed while the dimension of the initial matrix for the methods using an LP solver and the apparently better behaviour of the combinatorial method (with a fixed number of iterations).	
	\begin{figure} 
	\begin{scriptsize}
		\centering 
		\def\svgwidth{.8\columnwidth} 
		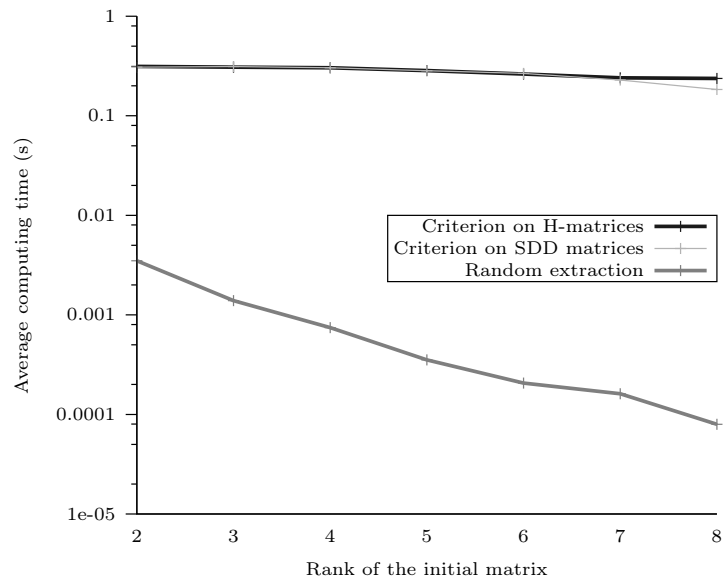
		\caption{Average computing time (in seconds) over iterations of the LP solver for the search of H-sub-matrices and strictly diagonally dominant sub-matrices, and for the random sub-matrix extraction for an interval matrix of size $8\times 8$ as a function of the minimum expected rank.}
		\label{fig:times2}
	\end{scriptsize}		
	\end{figure}
	Figure~\ref{fig:times2} shows average computing times of the LP solver for the search of H-sub-matrices and strictly diagonally dominant sub-matrices, and for the random sub-matrix extraction for an interval square matrix of dimension 8.
	In this case the matrix created randomly is constructed with a known minimum rank $r$ with $2 \leqslant r \leqslant 8$ that is the size of the interval H sub-matrix created.
	Using the LP solver, the execution time is more or less constant.
	We stress here the fact that the method using random extraction becomes faster for larger ranks because they are detected very early with Algorithm~\ref{algoCombi} and because the matrix can contain components equal to 0, the method using H-matrices in the case where $u$ is fixed as presented can no longer be applied since a division by 0 can occur in the linear program \eqref{eq:lp2}.
	
	Finally two last experiments are provided: they use a different technique to construct the interval matrix.
	To test the efficiency of the proposed algorithms, matrices with known rank $r$ are built.
	For that purpose, a triangular matrix $A$ (with $(A)_{ii} \neq 0,\ 1 \leqslant i \leqslant r$) is first created.
	It is the upper left part of a matrix  $J \in \R^{n\times n}$ with $n \geqslant r$.
	The $(n-r)$ remaining columns of $J$ are created as linear combinations of the $r$ first columns.
	Then a sequence of rotations are applied to the matrix $J$ (here we applied $n+1$ rotations).
	An angle $\theta_i$ is chosen randomly with $0 \leqslant \theta_i \leqslant \pi$ to compose the rotation matrix.
	A pair $(k,l)$ with $k \neq l$ and $1 \leqslant k, l \leqslant n$ is chosen randomly and the rotation is applied for coordinates $(k, l)$ and $(l, k)$.
	The matrix $A$ constructed in this way is not interval.
	
	Figure~\ref{fig:ranks} shows the results for a first experiment using this matrix construction which has the advantage to let us know the expected rank of the results.
	Using Theorem~\ref{theo:levy} as a criteria for sub-matrix extraction generally leads to a sub-matrix of dimension less than the rank of the matrix that is worse than the combinatorial method.
	\begin{figure} 
	\begin{scriptsize}
		\centering 
		\def\svgwidth{0.8\columnwidth} 
		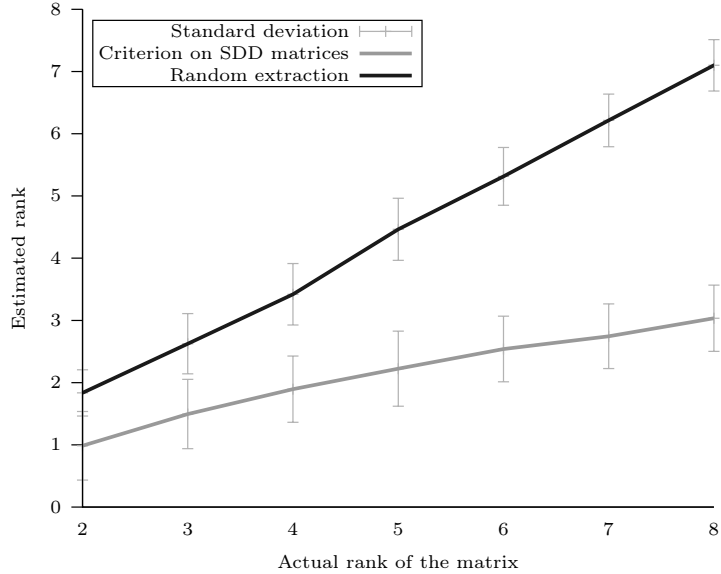
		\caption{Execution from a matrix $J \in \R^{8\times 8}$ of constant rank $i = 2, \dots, 8$. Results are average values over 200 computations.}
		\label{fig:ranks}
			\end{scriptsize}
	\end{figure}
	
	The matrix construction can be used to show the impact of the interval width of the components of the matrix on the method efficiency.
	For the last experiment, the same matrix construction than the previous one is used except that the components of the matrix constructed are thickened to obtain an interval matrix (the interval $\inter{-0.25, 0.25}$ is added to each component of the matrix for one experiment and $\inter{-0.5, 0.5}$ for another).
	Results are shown in Figure~\ref{fig:ranks2}.
	
	\begin{figure} 
	\begin{scriptsize}
		\centering 
		\def\svgwidth{0.8\columnwidth} 
		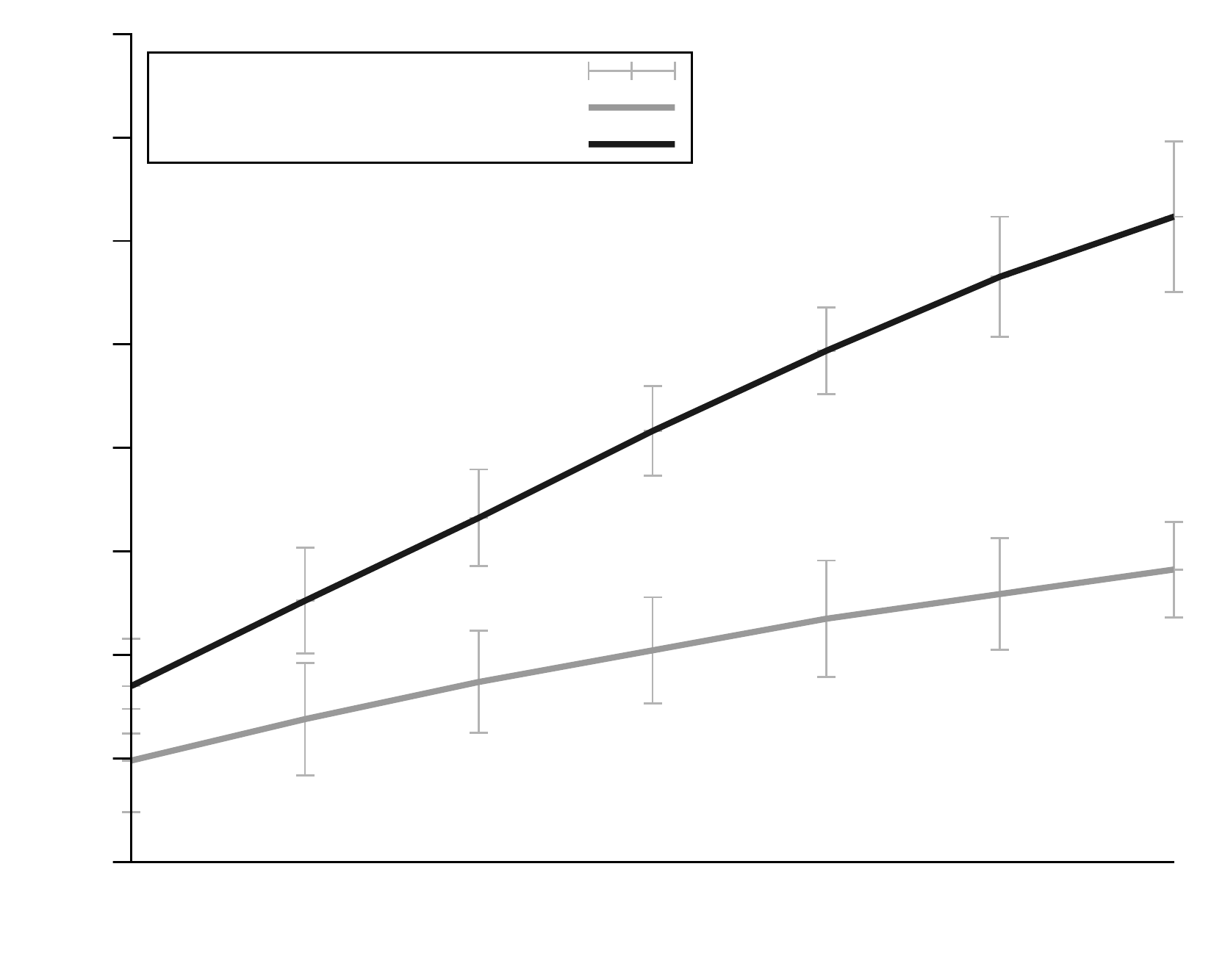
		\def\svgwidth{0.8\columnwidth} 
		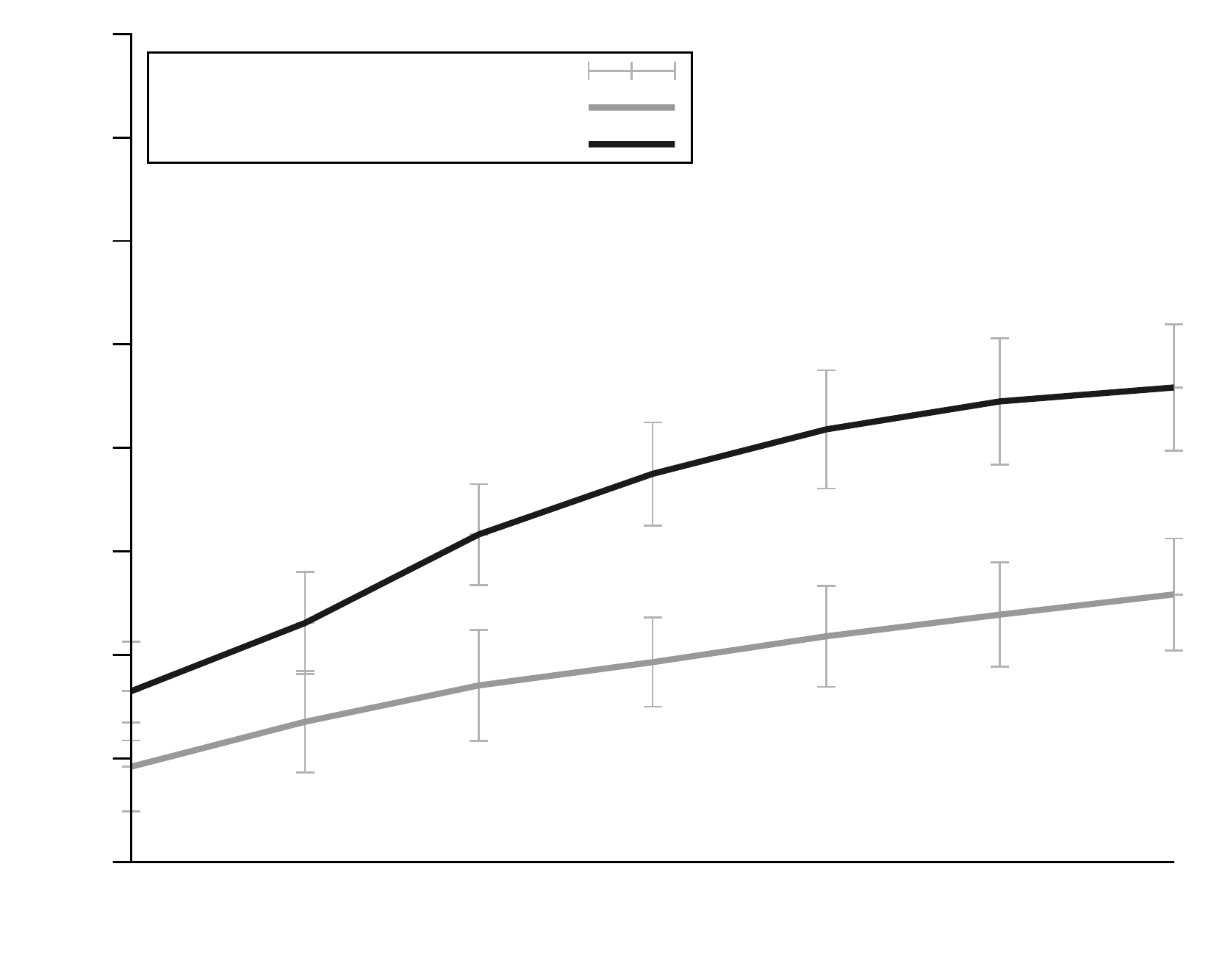
		\caption{Estimate of the rank of an interval matrix $\vA \in \IR^{8\times 8}$ of constant rank $i = 1, \dots, 8$. Upper picture: $\inter{-0.25, 0.25}$ is added to each component, lower picture: $\inter{-0.5, 0.5}$ is added to each component  Results are average values over 200 evaluations. Standard deviation is provided.}
		\label{fig:ranks2}
			\end{scriptsize}
	\end{figure}
	
	These experiments show that the combinatorial method is better than the method extracting strictly diagonally dominant sub-matrices.
	This is due to Theorem~\ref{thm:regular} which can detect a bigger subset of regular matrices than just strictly diagonally dominant ones.	
	
\subsubsection*{Limitations}
	As previously mentioned, this method cannot guarantee to obtain the sub-matrix of maximum rank.
	For a given function $f : \mathcal{D} \subseteq \R^m \rightarrow \R^n$ of constant rank $r$, only $p \leqslant r$ components can be detected from the LP program \eqref{eq:lp}.
	An inner approximation will be obtained only for these $p$ components of $f$.
	
\section{Computation of an inner approximation of the range of a vector-valued function}	
	\label{sec:app}
	This section shows some results of the computation of inner approximations of immersions and submersions.
	The functions considered in these examples satisfy \eqref{eq:inv} with $r$ known.

\subsection{Immersion}
 	\label{subsec:imm}
	Consider the problem of finding the range of the function 
	\begin{align}
		f:\R^2	&	\rightarrow \R^3\nonumber\\
		(u,v)		&	\mapsto	
		\left\lbrace
			\begin{array}{l}
				f_1(u,v) = \cos(u)\cos(v)\\
				f_2(u,v) = \sin(u)\cos(v)\\
				f_3(u,v) = \sin(v)		
			\end{array}
		\right.
		\label{ex:immersion}	
	\end{align}
	over the box $(\vu, \vv) = (\left[\frac{3\pi}{2}+\tau,2\pi-\tau\right];\left[\tau,\frac{\pi}{2}-\tau\right]),\ \tau>0$.
	$f$ is of constant rank 2 in 	$(\vu, \vv)$.
	The rank is equal to the dimension of the domain of $f$, it is then an \emph{immersion}.
	Corollary~\ref{cor:rang} can be used to get an inner approximation of the range of two components of $f$.
	Here we compute the range of the two first components, but the two last or the first and the last components could also be considered.
	
	\begin{figure} 
	\begin{scriptsize}
		\centering 
		\def\svgwidth{0.6\columnwidth} 
		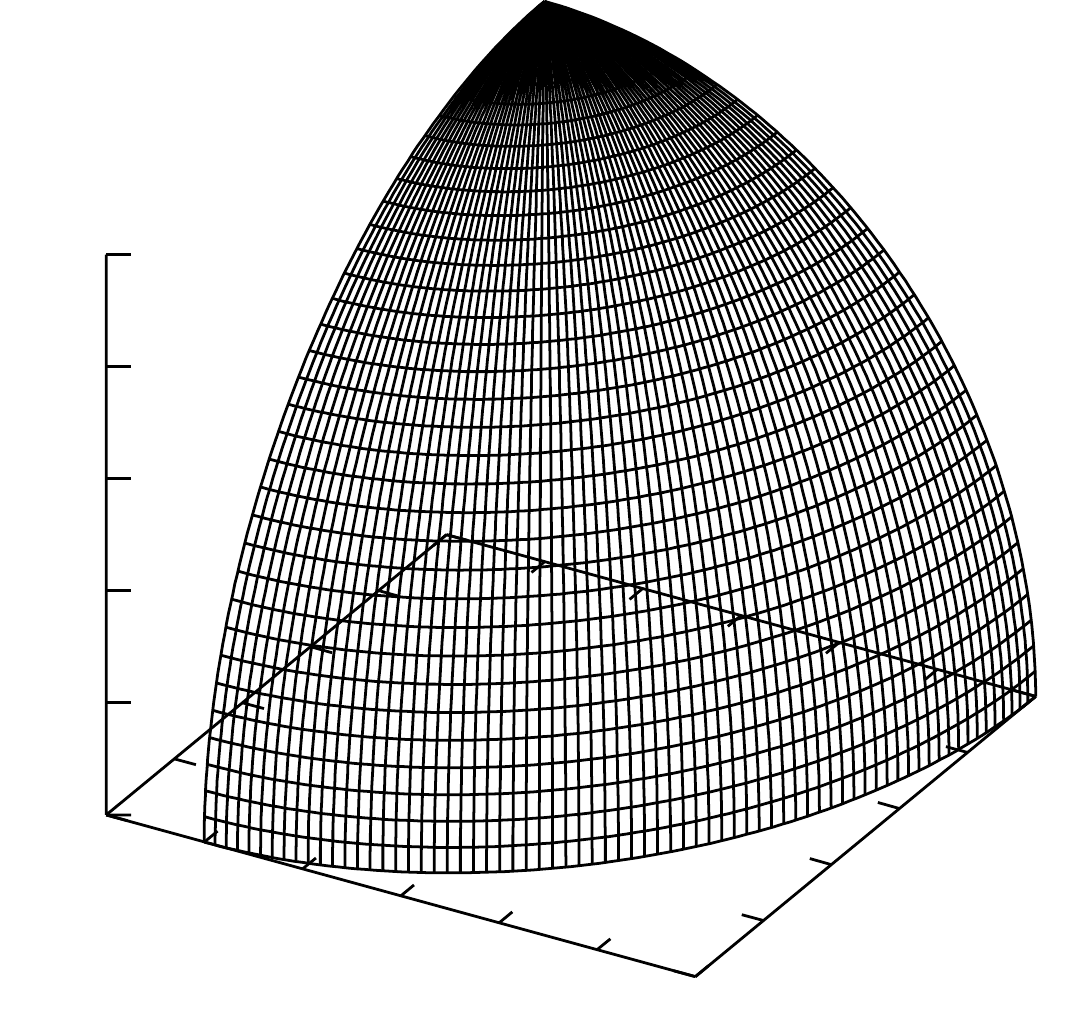
		\caption{Example 1: Range of the immersion defined in (\ref{ex:immersion}) for the initial domain $(\vu, \vv) = (\left[\frac{3\pi}{2}+\tau,2\pi-\tau\right];\left[\tau,\frac{\pi}{2}-\tau\right]),\ \tau>0$.}
		\label{fig:boule}
			\end{scriptsize}
	\end{figure}
	Figure~\ref{fig:boule} represents the image of $(\vu, \vv)$ by $f$ which has no volume in its co-domain $\R^3$.
	Results of the computation of an inner approximation of this range are shown in Figure~\ref{fig:result} for different values of $\varepsilon$ in Algorithm~\ref{algo:bis}.
	The smaller the $\varepsilon$, the more accurate the inner approximation and the longer the computing time.
	In Figure~\ref{fig:result}, empty boxes in gray represent boxes Algorithm~\ref{algo:bis} was unable to prove to be in the range.
	Black boxes all belong to the range.
	
	\begin{figure} 
	\begin{tiny}
		\centering 
		\begin{tabular}{ccc}
			\def\svgwidth{0.3\columnwidth} 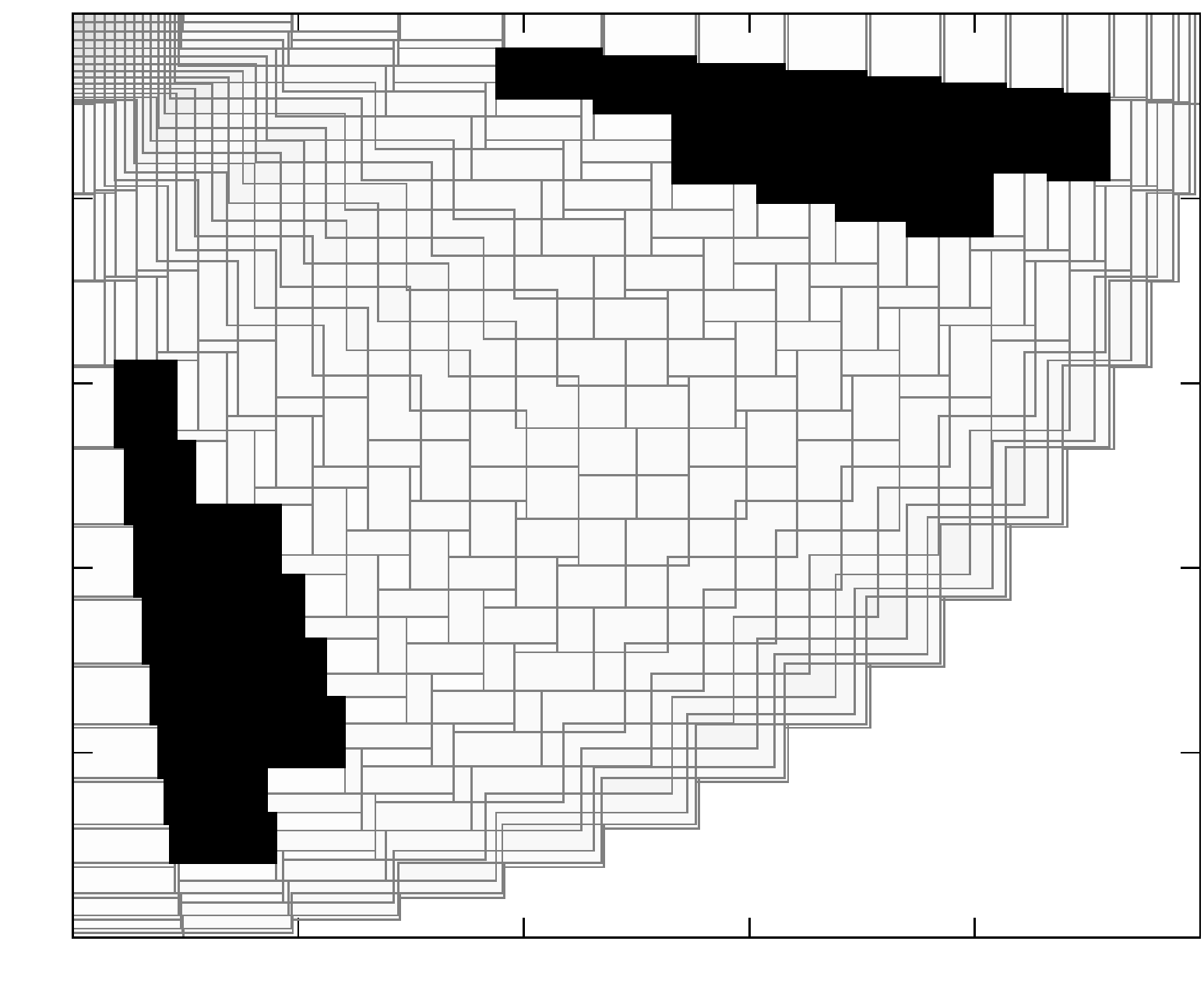&
			\def\svgwidth{0.3\columnwidth} 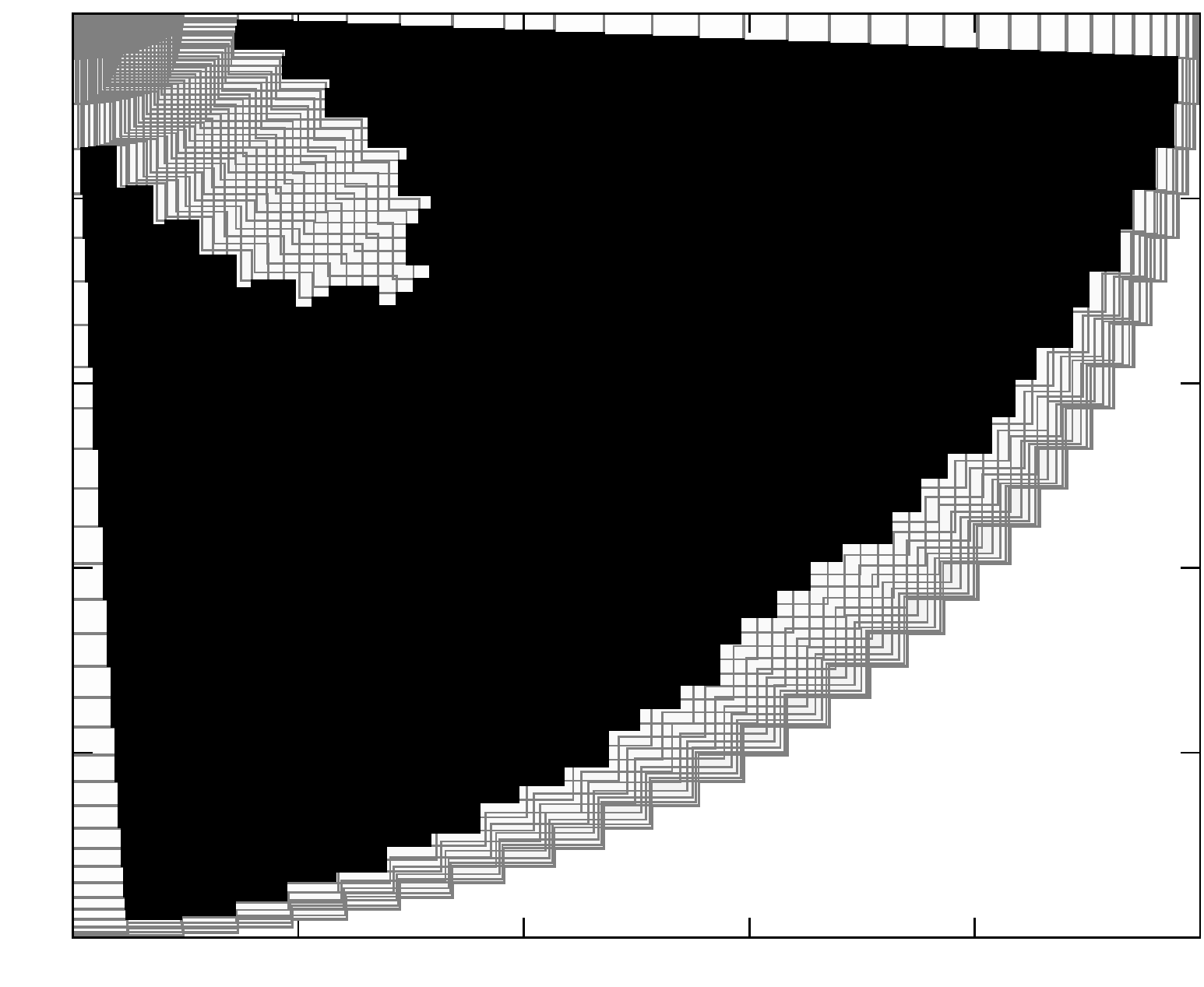& 
			\def\svgwidth{0.3\columnwidth} 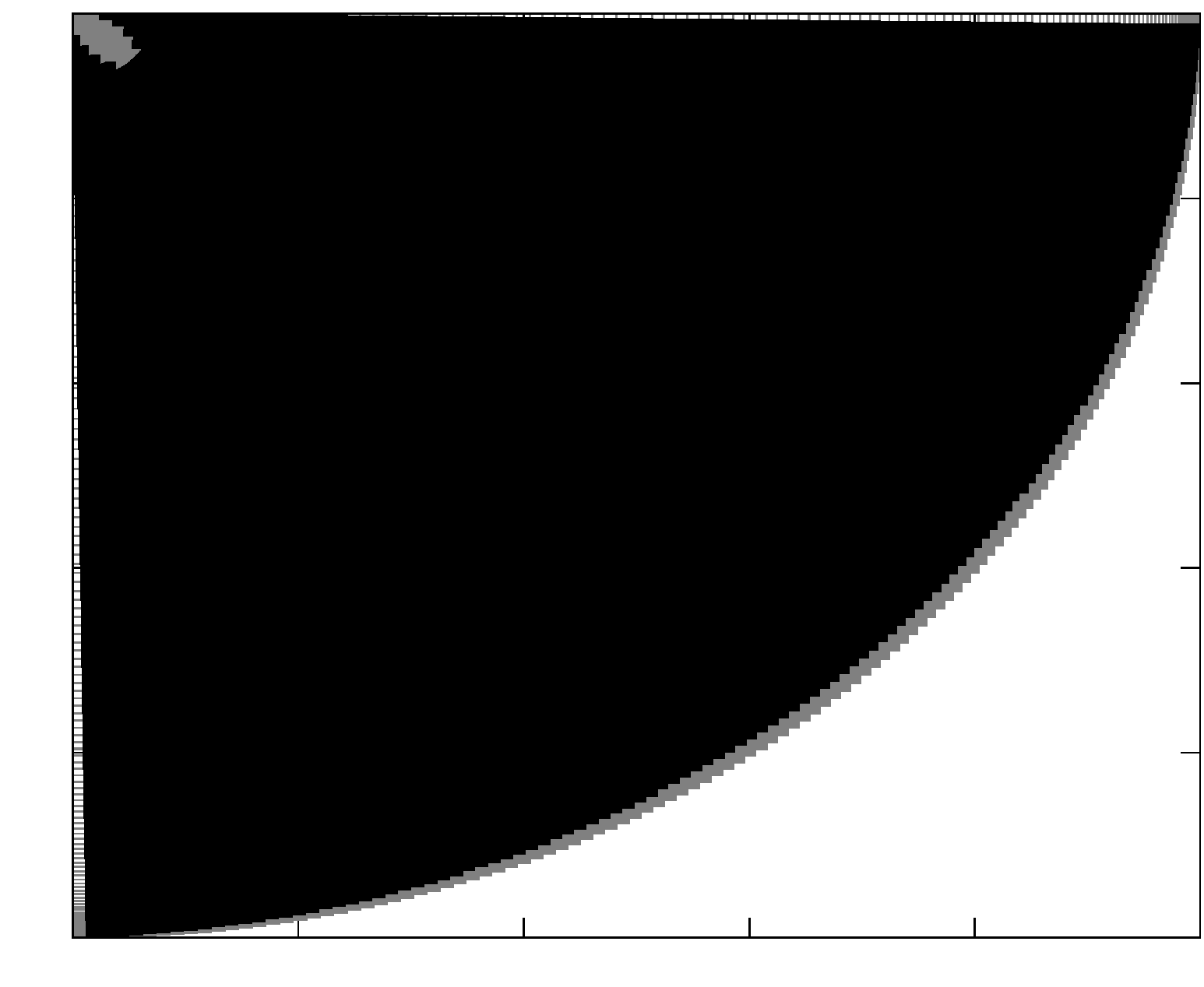
		\end{tabular} 
		\caption{Example 1: Results of the computation of an inner approximation for different values of parameter $\varepsilon$ (0.1, 0.06, and 0.02).}
		\label{fig:result}
	\end{tiny}	
	\end{figure}

	The left part of Figure~\ref{fig:result} is for $\varepsilon = 0.1$, the middle part for $\varepsilon = 0.06$ and the last part  for $\varepsilon = 0.02$.
	The computing times are respectively 0.026 s, 0.10 s, and 0.64 s.
	
	\subsection{Submersion}
	Consider now the computation of an inner approximation of the range of the function
	\begin{align}
		f: (\vx, \vy, \vz) \subseteq \R^3	&	\rightarrow \R^2\nonumber\\
		(x,y,z)		&	\mapsto	
		\left\lbrace
			\begin{array}{l}
				(x+r\cos(z))\cos(y)\\
				(x+r\sin(z))\sin(y)
			\end{array}
		\right.	
		\label{ex:sub}
	\end{align}
	with $(\vx, \vy, \vz) = (\inter{2, 4.5}, \inter{0, 2\pi - \tau}, \inter{0, 2 \pi - \tau})$, $\tau = 10^{-3}$.
	Figure~\ref{fig:donut} represents the range of $f$ and Figure~\ref{fig:result2} represents different computations of an inner approximation according to the parameter $\varepsilon$ in Algorithm~\ref{algo:bis}.
	On Figure~\ref{fig:result2}a $\varepsilon = 0.5$ and it took 0.18s to get the result.
	For Figure~\ref{fig:result2}b, $\varepsilon = 0.3$ and computation time is  6.25s.
	In Figure~\ref{fig:result2}c, it took 145.53s with $\varepsilon = 0.1$ to get these results.
	Finally Figure~\ref{fig:result2}~d is for $\varepsilon = 0.05$ and the computing time is 838.67s.
	The time needed for computation is longer for this experiment than for the previous one on immersion.
	It is due to the fact that the Jacobian of $f$ is not of full rank in the entire domain $(\inter{2, 4.5}, \inter{0, 2\pi - \tau}, \inter{0, 2 \pi - \tau})$.
	It is why an area (cf. Figure~\ref{fig:result2}d remains out of the range of $f$.
	
	\begin{figure} 
	\begin{scriptsize}
		\centering 
		\def\svgwidth{0.3\columnwidth} 
		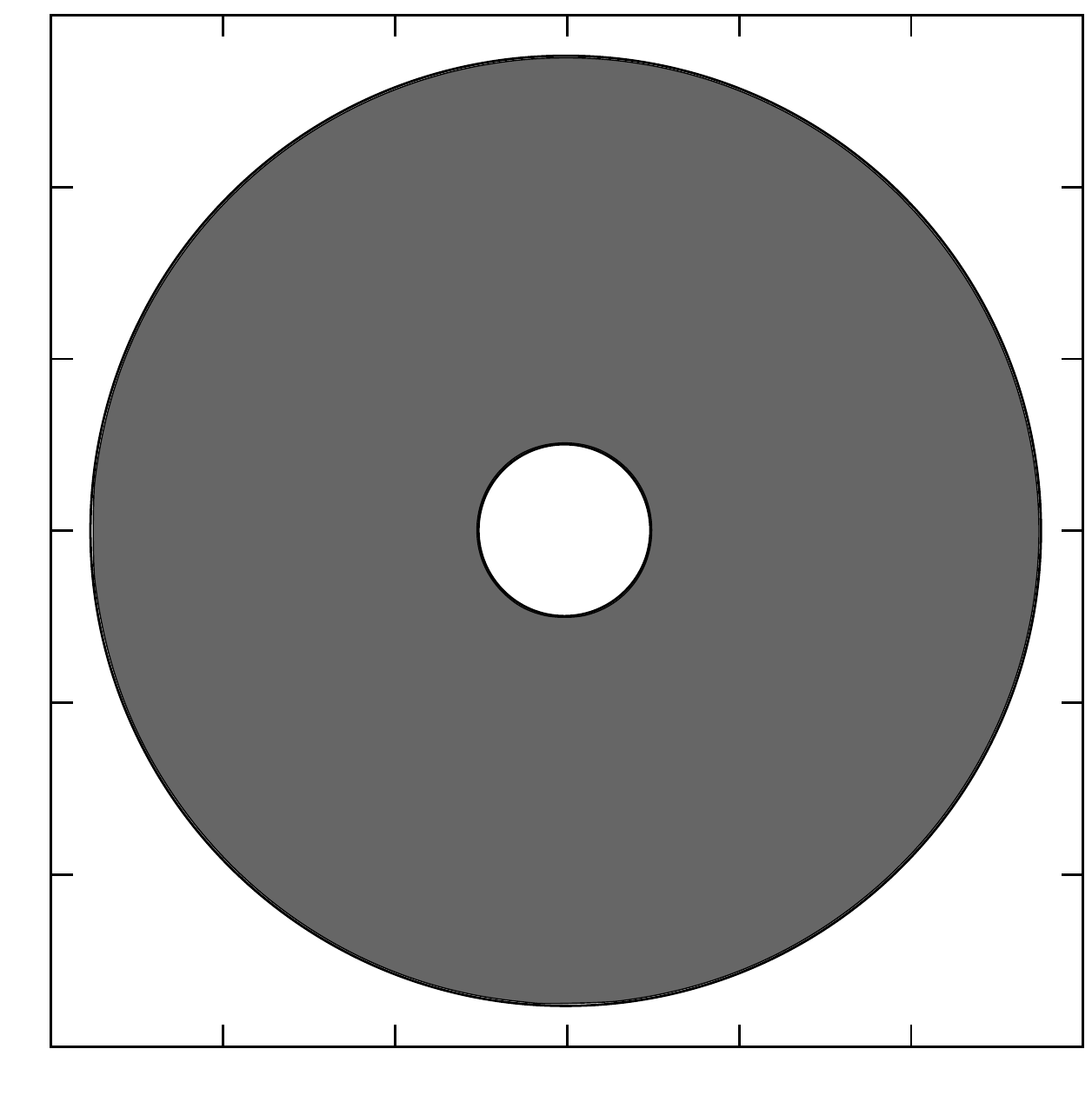
		\caption{Example 2: Range of the submersion defined in (\ref{ex:sub}) for the initial domain $(\vx, \vy, \vz) = (\inter{2, 4.5}, \inter{0, 2\pi - \tau}, \inter{0, 2 \pi - \tau})$, $\tau = 10^{-3}$.}
		\label{fig:donut}
			\end{scriptsize}
	\end{figure}
	\begin{figure}
	\begin{scriptsize}
		\centering 
		\begin{tabular}{cc}
			\textbf{a)}\def\svgwidth{0.4\columnwidth} 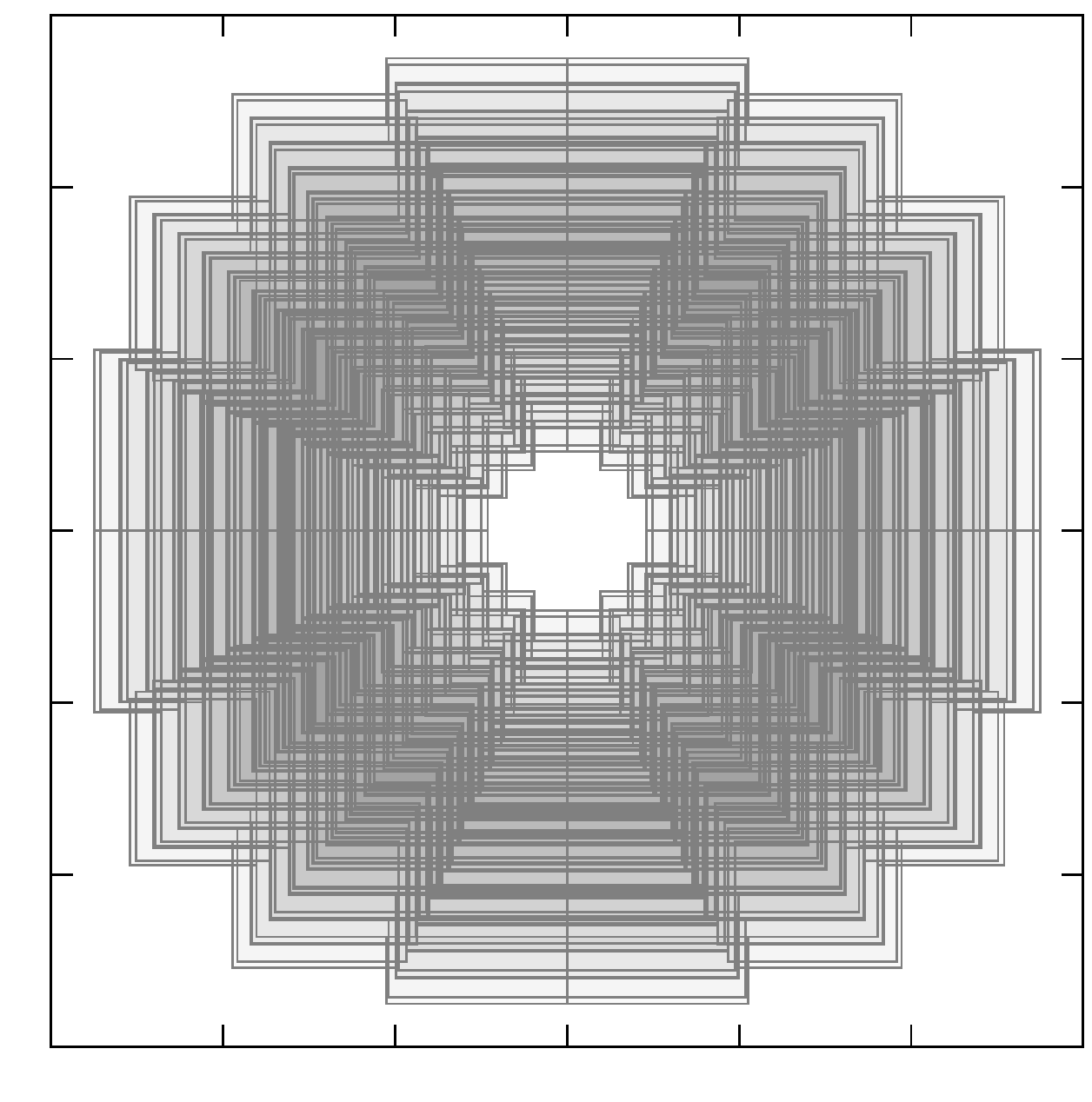&
			\textbf{b)}\def\svgwidth{0.4\columnwidth} 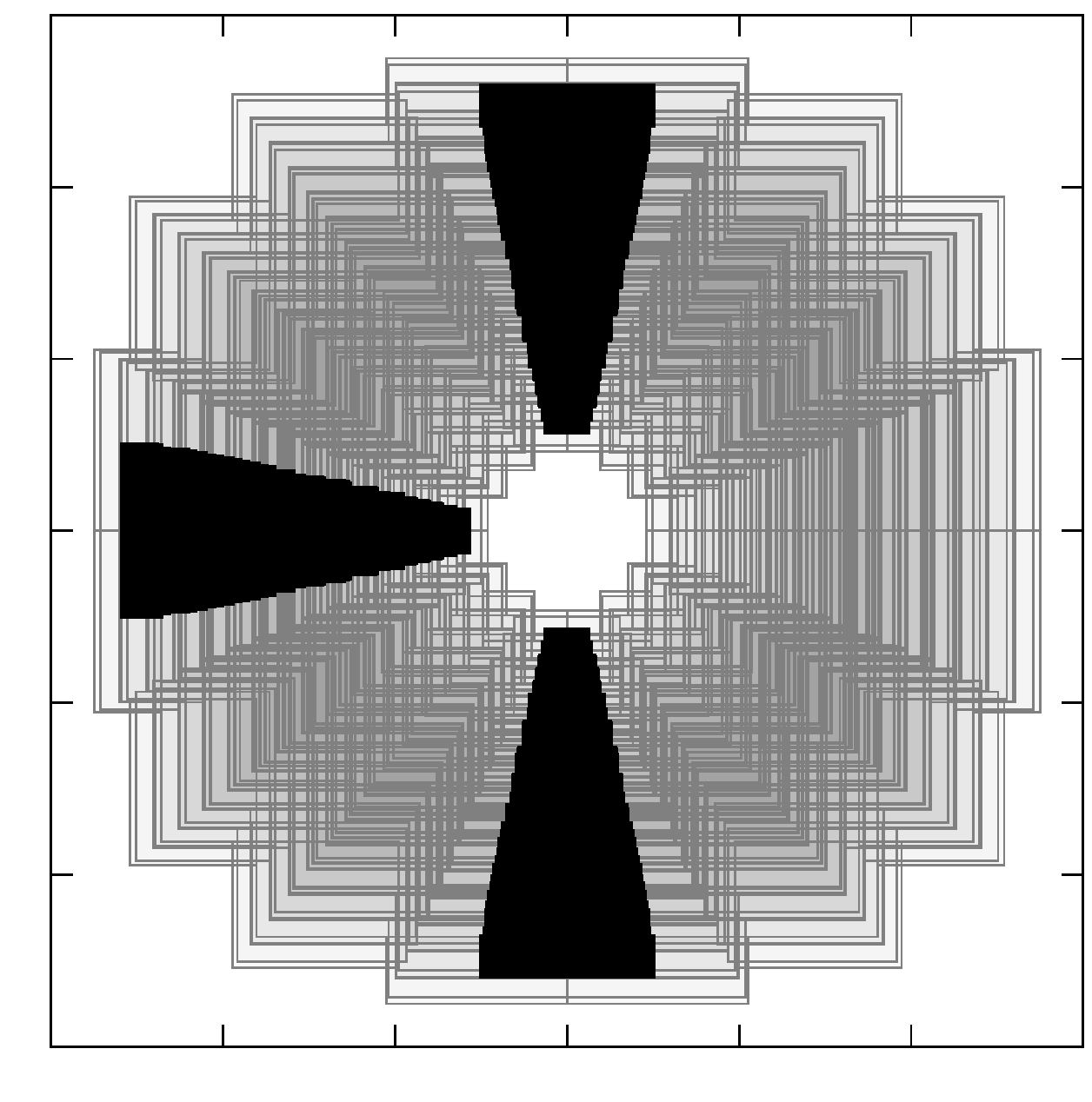\\
			\textbf{c)}\def\svgwidth{0.4\columnwidth} 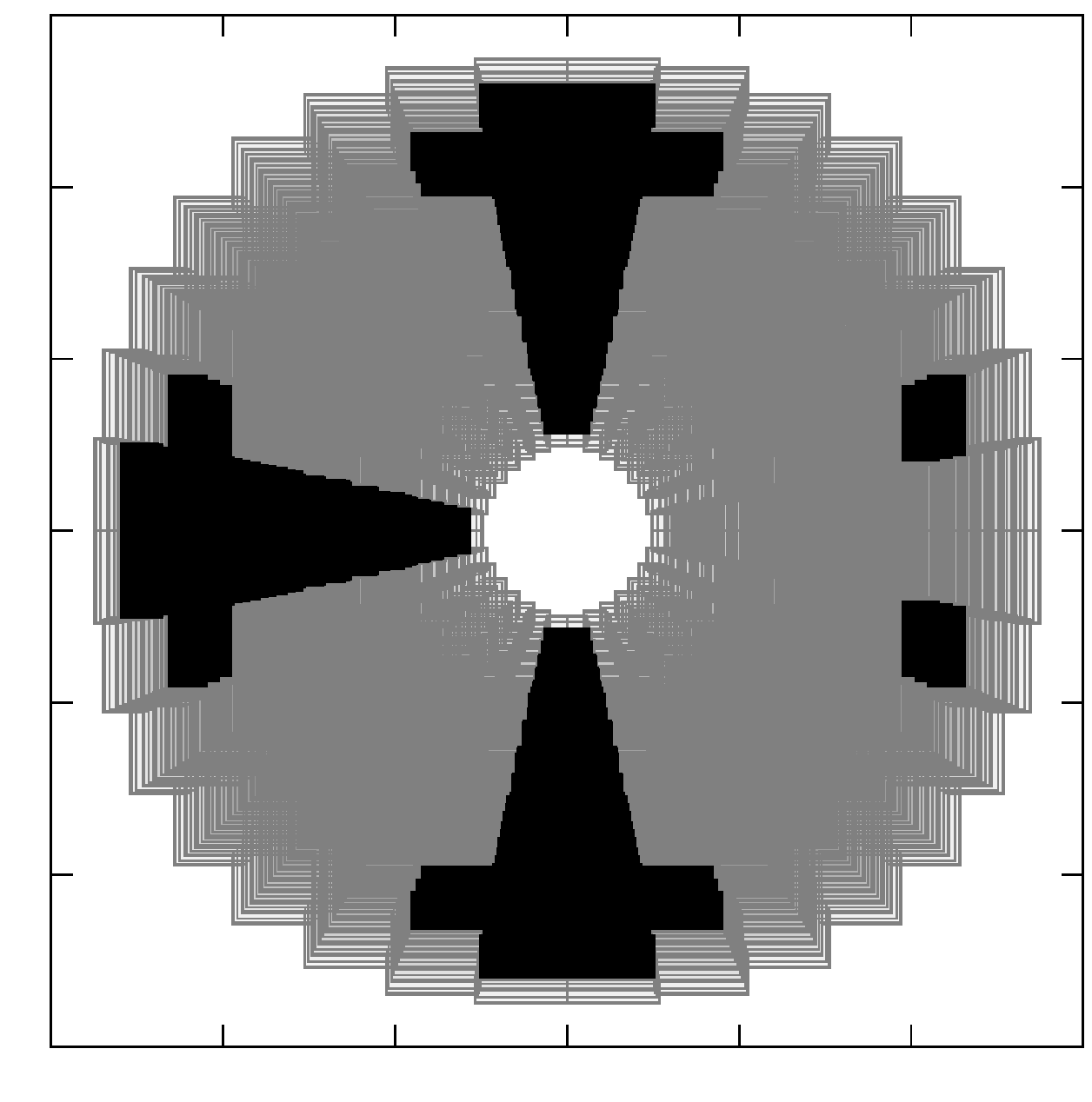&
			\textbf{d)}\def\svgwidth{0.52\columnwidth} 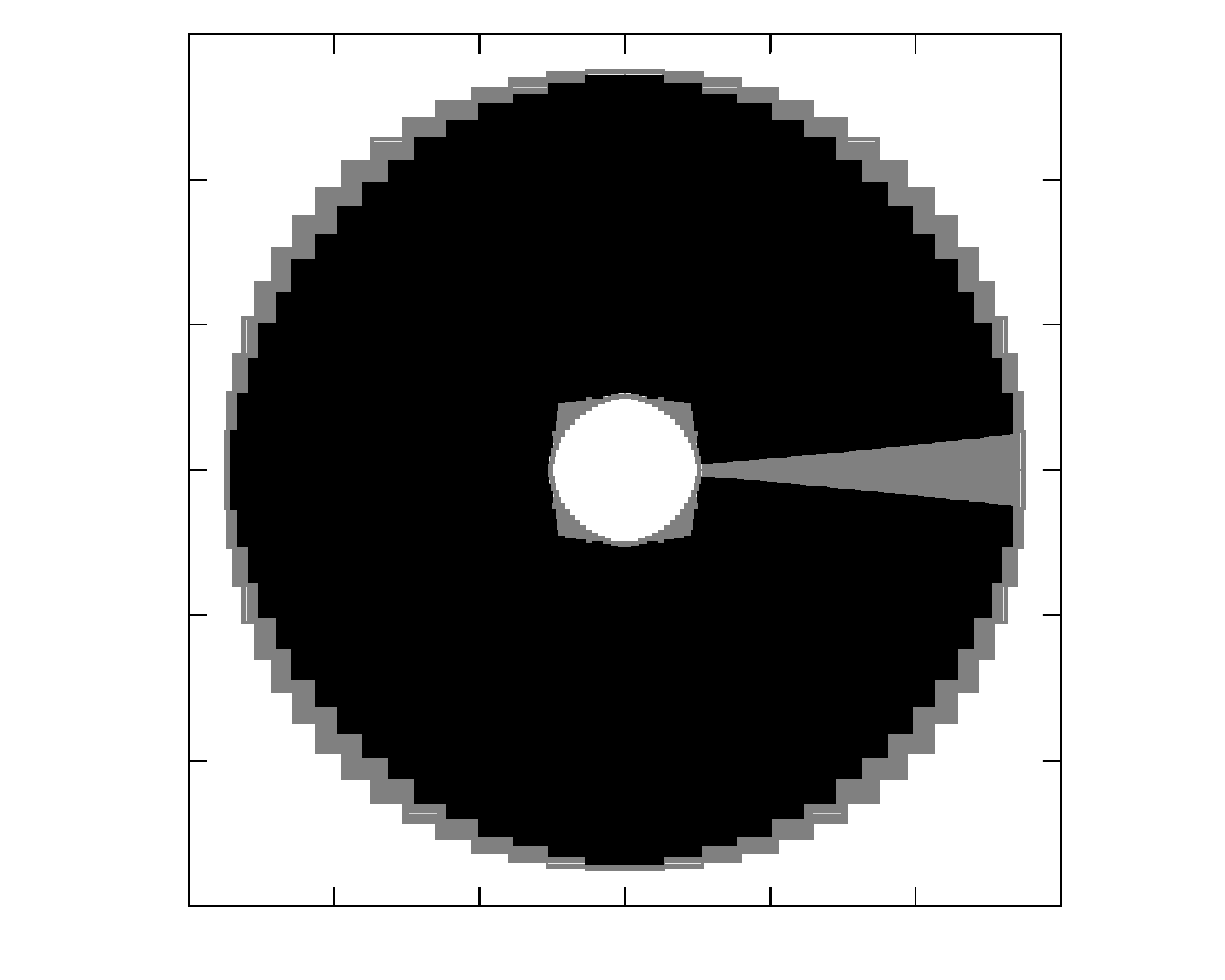
		\end{tabular} 
		\caption{Example 2: Results of computation of an inner approximation according to the parameter $\varepsilon$ in Algorithm~\ref{algo:bis} ($\varepsilon = 0.5$, 0.3, 0.1 and 0.05).}
		\label{fig:result2}
	\end{scriptsize}	
	\end{figure}

\section{Conclusion}

	Goldsztejn and Jaulin in~\cite{Goldsztejn2010} proposed a way to compute an inner approximation of the range of a vector-valued function.	
	This paper provides an algorithm to evaluate inner approximations of the range of vector-valued functions without restriction on the dimension of its domain and co-domain.
	Using the proposed algorithm, one is able, for functions from $\R^m$ to $\R^n$ to evaluate an inner approximation of the projection of the range of $f$ on at most $r$ components, where $r$ is the rang of the Jacobian matrix of $f$.
	In the general case, this rank $r$ is unknown a priori, it is thus necessary to develop several techniques to extract a sub-matrix of maximal rank from a given interval matrix.
	The restriction of this method is providing an inner approximation of at most $r$ components of the function if this function has a constant rank $r$.
	
	\bibliographystyle{plain}
	\bibliography{references}	
\end{document}